%% file: main.tex
\documentclass[10pt,reqno]{amsart}
\usepackage[utf8]{inputenc}
 \usepackage{float}
 \usepackage{multicol}
\usepackage{fullpage, tikz-cd} % add this since it gets rid of the super large margins
\usepackage{graphicx}
\usepackage{enumitem}
\usepackage{amsmath,amssymb,amsthm, amsfonts,verbatim, pdfpages, wrapfig, blindtext,mathtools}
\usepackage[colorlinks=true,hyperindex, linkcolor=blue, pagebackref=false, citecolor=cyan]{hyperref}
\usepackage{tikz,xcolor}
\usepackage{cleveref}
\usepackage{blkarray}
\usepackage{multirow, comment}
\newcommand{\bigzero}{\mbox{\normalfont\Large\bfseries 0}}
\newcommand{\rvline}{\hspace*{-\arraycolsep}\vline\hspace*{-\arraycolsep}}

\usepackage{nicematrix}
\usepackage{tikz}
\usetikzlibrary{fit}
\usetikzlibrary{shapes.geometric}

\newcommand{\Z}{\mathbb{Z}}

\DeclareMathOperator{\Cone}{Cone}

\DeclareMathOperator{\coker}{coker}
\DeclareMathOperator{\im}{im}
\DeclareMathOperator{\diag}{diag}

\def\onesvector{\mathbf{1}}
\def\zerovector{\mathbf{0}}

\newcommand{\CT}[2]{\mathrm{CT}(#1,#2)}

\newcommand{\Coneplus}[1]{\mathrm{Cone}^{(+#1)}}

        \usepackage{mathtools}

\usepackage{graphicx}

\newtheorem{theorem}{Theorem}[section]
\newtheorem{lemma}[theorem]{Lemma}

\author{
  Dorian Smith \textsuperscript{\textdagger}}
\thanks{\textsuperscript{\textdagger} University of Minnesota Twin Cities,
Department of Mathematics, smi01055@umn.edu}
\address[D.~Smith]{University of Minnesota Twin Cities,
Department of Mathematics,
Minneapolis, MN 55455}

\email[D.~Smith]{\textcolor{blue}{\href{mailto:smi01055@umn.edu}{smi01055@umn.edu}}}

\begin{document}

\title{The Sandpile Group of a  Cone Over a Bi-Coconut Tree}

\begin{abstract}
 The sandpile group of a connected graph is a finite abelian group whose cardinality is the number of spanning trees in the graph.  We compute the spanning tree number and sandpile group structure for the cone over a bi-coconut tree, generalizing work of Reiner and Smith on the cone over a coconut tree.  We also answer one of their questions,
 by exhibiting a family of trees whose sandpile groups are all cyclic but their number of leaves grows without bound.
\end{abstract}

\maketitle

\section{Introduction and Preliminaries}

 In this paper,
we study two important and related invariants of  cones over bi-coconut trees: the spanning tree number and the sandpile group.
 We first explain some motivation and background and then present our
main results.

In \cite{Reiner2024SandpileGF}, the authors asked:
\begin{quote}
How does the ``shape" of the graph $G$ relate to the structure of its
sandpile group $K(G)$, including
the cardinality $|K(G)|=\tau(G)$ which is the number of spanning trees, and the minimal number of generators $\mu(G)$ for $K(G)$? 
\end{quote}
To explore the answer to this  question they examined the graphs which are cones over trees. In particular, they studied the cone over a coconut tree, a family which illustrates several interesting features,
and includes the extreme cases of cones over path graphs and star graphs. In this paper we further explore this motivating question.

The {\it cone over a graph}, denoted $\Cone(G)$, 
is obtained from a graph $G$ on vertex set $\{v_1,\ldots,v_n\}$ by adding an $(n+1)^{st}$ {\it cone vertex} $v_0$, with one edge $\{v_0,v_i\}$
for $i=1,2,\ldots,n.$ 
For $p,s$ positive integers, the {\it coconut tree} $\CT ps,$ is the graph  obtained from a
path on $p$ vertices by adding $s$  leaf vertices to  one endvertex of the path.  Here is $\CT 54$:
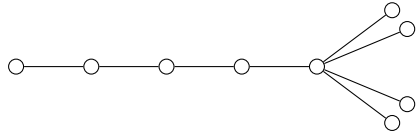
\vspace{.75cm}

\begin{center}
\input{drawing1}
\end{center}

\noindent
More generally, for $p,s_1,s_2$ positive integers, the {\it bi-coconut tree}, denoted $T(p,s_1,s_2)$, is obtained from a path on $p$ vertices by adding $s_1, s_2$ leaf vertices respectively to each endvertex of the path. Here is $T(5,3,4)$:
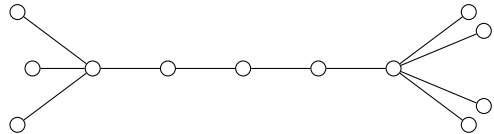
\vspace{.75cm}  

\begin{center}
\input{drawing2}
\end{center}

\noindent
Note that a coconut tree is a special case of a bi-coconut tree: $\CT ps =T(p-1,1,s)$.

\subsection{Sandpile groups and Laplacians}

Let  $G = (V, E)$ be a finite undirected graph on vertex set $V$ of cardinality $n$, and edge set $E$. The {\it sandpile group} $K(G)$ (defined below) is an interesting and subtle isomorphism invariant of $G$, that comes in the form of a finite abelian group whose cardinality is the number of spanning trees, $\tau(G).$

We consider graphs that are connected. Parallel edges are allowed, but no self-loops.  The {\it Laplacian matrix} of $G$ is defined to be the symmetric matrix $L_G := D - A$ where  $D$ is the diagonal matrix of vertex degrees in $G$, and $A$ denotes its $\{0,1\}$ adjacency matrix \cite[Chapter 2]{Klivans}. 
The $|V| \times |V|$  graph Laplacian,  whose rows and columns are both indexed by
the vertices $v$ in $V$, can be viewed as   a $\mathbb{Z}$-module map $L_G: \mathbb{Z}^V \rightarrow \mathbb{Z}^V$. One way to define the {\it sandpile group} $K(G)$ is as the torsion group of the integer cokernel of $L_G$.  Since $G$ is connected, this will imply $L(G)$ has
$1$-dimensional nullspace, spanned by the all ones vector $\onesvector_n$.  Hence
$L(G)$ has rank $n-1$, where $n$ is the number of vertices, so that 
the cokernel of $L_G$ is
\begin{equation*} \label{cokernel} 
\coker( L_G) = \mathbb{Z}^n / \im(L_G) \cong \mathbb{Z} \oplus K(G). 
\end{equation*}
Alternatively, one can consider a {\it reduced Laplacian} $\overline{L_G}$, obtained by deleting the $i^{th}$ row and $i^{th}$ column of the  Laplacian for any $i=1,2,\ldots,n$, 
and then one has
\begin{equation}
    \label{sandpile-group-presentation}
K(G):= \coker (\overline{L_G})= \mathbb{Z}^{n-1} / \im(\overline{L_G}).
\end{equation}

In practice, the sandpile group of $G$ can be computed  by calculating  the {\it Smith normal form} of the Laplacian $L_G$ \cite{LORENZINI20081271}: there exist $n \times n$ invertible integer matrices $P$ and $Q$ for which
$PL_GQ = \mathrm{diag}(\lambda_1...\lambda_n)$
where
$\mathrm{diag}(\lambda_1...\lambda_n)$ is a diagonal matrix with integer entries $\lambda_i$, and $\lambda_{i-1}$ divides $\lambda_{i}.$  It can be shown that  
$K(G)=\mathbb{Z}^n / \im(L_G) $ is isomorphic to  $\bigoplus_{i=1}^{n} \mathbb{Z} / \lambda_i\mathbb{Z}.$

Kirchhoff's Matrix Theorem  \cite[\S 5.6]{Stanley}, states that 
$$
|K(G)| =  \det (\overline{L_G}) = \tau(G),
$$ 
the number of {\it spanning trees} of $G.$ 
Thus one can think of the structure of the
group $K(G)$ as a refinement of the spanning
tree number $\tau(G)$. For additional background on sandpile groups, see the texts by Corry and Perkinson \cite[Chap. 6]{CorryPerkinson} and Klivans \cite[Chap. 4]{Klivans}.

\subsection{Main results}

\vskip.1in
For each integer $s_1 \geq 0$, define a sequence $\{b_n\}_{ n=-2,-1,0,1,2,3,\ldots}$
where $b_n:=b_n^{(s_1)}$ is defined by 
this Fibonacci-style recurrence \eqref{b_n}, whose initial conditions \eqref{b_-2},\eqref{b_-1} depend upon $s_1$:
\begin{align}
    b_n &= b_{n-1} + b_{n-2} \label{b_n}\\
b_{-2} &= 2^{s_1} \label{b_-2}\\  
b_{-1} &= 2^{s_1-1}(s_1+2).\label{b_-1}
\end{align}
In particular, using the Fibonacci number convention where $F_1=F_2=1$, one has

\begin{equation}
\label{b-F-relation}
b_n=\begin{cases}
F_{n+3}&\text{ when }s_1=0,\\
F_{n+5}&\text{ when }s_1=1.\\
\end{cases}
\end{equation}
For integers $p \geq 1$, and $s_1,s_2 \geq 0$, define a further quantity
\begin{equation}
\label{t-definition}
t(p,s_1,s_2):= 2^{s_2-1}(2b_{2p-3}^{(s_1)} +s_2 b_{2p-4}^{(s_1)}).
\end{equation}

\begin{theorem}
\label{spanning-tree-number-theorem}
For $p \geq 1$ and $s_1,s_2 \geq 1$, the cone over the bi-coconut tree $T(p,s_1,s_2)$ has
$$
\tau(\Cone(T(p,s_1,s_2)))=t(p,s_1,s_2).
$$
\end{theorem}

The asymmetry of $s_1, s_2$ in the formula for $t(p,s_1,s_2)$
in Theorem~\ref{spanning-tree-number-theorem} is
resolved by the next result.

\begin{theorem}
\label{generating-function-theorem}
For $s_1,s_2 \geq 1$,  $t(p,s_1,s_2)$ has
ordinary generating function in $p$ given by
$$
\displaystyle\sum_{p=1}^\infty t(p, s_1, s_2)x^{p-1}
=2^{s_1+s_2-1} \cdot \displaystyle\frac{4 + 2(s_1+s_2)(1-x) +s_1s_2x}{1-3x+x^2}.
$$ 
In particular, this function is symmetric in $s_1,s_2$, so the same holds for $t(p,s_1,s_2)$.
\end{theorem}

The next result gives the sandpile group structure
for $\Cone(T(p,s_1,s_2)$.

\begin{theorem}
\label{sandpile-group-theorem}
For $p \geq 1$ and $s_1,s_2 \geq 1$, defining
\begin{equation}
\label{formula-for-a}
    a:=2^{1-s_1}(2\cdot b_{2p-3}^{(s_1)} +s_2b_{2p-4}^{(s_1)})
    =\frac{t(p,s_1,s_2)}{2^{s_1+s_2-2}},
\end{equation}
then the sandpile group for the cone over the
bi-coconut tree $T(p,s_1,s_2)$ has this description:
$$
K(\Cone(T(p,s_1,s_2)) \cong
\begin{cases}
    \Z_2^{s_1+s_2-2} \oplus \Z_a 
    & \text{ if } p \equiv 1 \bmod{3},\\
    \Z_2^{s_1+s_2-3} \oplus \Z_{2a}   
    & \text{ if }s_1   \text{ or } s_2 \text{ odd } \text{ and }p \not\equiv 1 \bmod{3},\\
    \Z_2^{s_1+s_2-4} \oplus \Z_4\oplus\Z_{a}  
    & \text{ if }s_1 \text{ and } s_2  \text{ even  and }p \not\equiv 1 \bmod{3}.
\end{cases}
$$
\end{theorem}
\noindent
Theorems~\ref{spanning-tree-number-theorem}
and \ref{sandpile-group-theorem} specialize
to the analogous results from \cite{Reiner2024SandpileGF},
using \eqref{b-F-relation}. In particular, we will use
a result from \cite{Reiner2024SandpileGF}, namely
Theorem~\ref{old-theorem_in_RS} below as a lemma in our proof of
Theorem~\ref{spanning-tree-number-theorem}.
Given a vertex $v$ in a tree $T$, define the graph $\Coneplus{v}(T)$ to be the graph obtained from $\Cone(T)$ with cone vertex $v_0$ by creating a second parallel copy of the edge $\{v_0,v\}$.

\begin{theorem}  
\label{old-theorem_in_RS}
\cite{Reiner2024SandpileGF}
For $p,s \geq 1$, the cone over the coconut tree $\CT ps$
has 

\begin{align}
\tau (\Cone( \CT ps)) &= 2^{s-1} \cdot (2{F}_{2p+1} + (s-2){F}_{2p-1}),\label{eq:1}\\
\tau (\Coneplus{\pi_1}(\CT ps)) &= 2^{s-1} \cdot (2{F}_{2p+2} + (s-2){F}_{2p})\label{eq:2} 
\end{align}
where $\pi_1$ is the endpoint vertex of the path
inside $\CT ps$
which is {\it not} attached to the $s$ leaves, and $F_n$ is the $n$th Fibonacci number.
\end{theorem}

\noindent
We will also need a fact used in \cite{Reiner2024SandpileGF}
about spanning tree numbers $\tau(G)$ for cone graphs $G$, a special case of the {\it deletion-contraction recurrence} $\tau(G)=\tau(G/e)+\tau(G\setminus e)$.

\begin{lemma}\cite[Prop.~3.5 at $y=1$]{Reiner2024SandpileGF}\label{old-theorem-used-as-lemma} For $T$ a tree, $\ell$ a leaf vertex of $T$, and $v$ the unique neighbor vertex of $\ell$ in $T$, one has
\begin{align} 
%\label{specialized-leaf-del-con}
 \tau(\Cone(T))
&=\tau(\Cone(T -\ell))
 + \tau({\Coneplus{v}(T-\ell)}),\label{thm 1.5-1}\\
\tau(\Coneplus{\ell}(T)))
&=\tau(\Cone(T))
 + \tau(\Coneplus{v}(T-\ell)).\label{thm 1.5-2}
\end{align}
\end{lemma}

The proof of Theorem~\ref{sandpile-group-theorem} also uses as a crucial tool the following result from \cite{Reiner2024SandpileGF} on generators for
$K(G)$ when $G=\Cone(T)$ for any tree $T$. Let $\{e_v\}$ be the standard basis vectors
for $\Z^{n-1}$, and denote by
$\{ \overline{e}_v \}$ their images within
the quotient presentation
of the sandpile group 
$K(G)=\Z^{n-1}/\im \overline{L}_G$
from \eqref{sandpile-group-presentation}.

\begin{theorem}
\label{generators} \cite{Reiner2024SandpileGF} 
Choosing the cone vertex of $\Cone(T)$ as the root 
vertex $v_0$, the group $K(\Cone(T))$
is generated by 
$
\{\overline{e}_v: \text{ leaves }v \neq v_1\}
$
for any choice of a leaf vertex $v_1$ in $T$.
\end{theorem}

In particular, the number of generators $\mu(\Cone(T))$ for the sandpile group $K(\Cone(T))$ is bounded above by $\ell(T)-1$, where $\ell(T)$ is the number of leaves of $T$.  For the coconut and
bi-coconut trees, this bound turns out to be nearly sharp, but the authors in \cite{Reiner2024SandpileGF} suspected that this
bound should be far from sharp in general for trees.  They asked the following question in
\cite[Question 2.2]{Reiner2024SandpileGF}:
\begin{quote}
Are there trees $\{T_p\}_{p=1,2,\ldots}$ all having $\mu(\Cone(T_p))=1$
with $\ell(T_p)$ growing arbitrarily large?
\end{quote}
Our last result answers this affirmatively, using the
{\it left comb tree} $T_p$ on $2p-1$ vertices; $T_6$ is shown here:

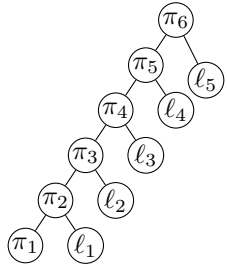
\begin{center}
\begin{tikzpicture}[scale=.4]
\tikzstyle{every node}=[draw,circle,minimum size=1pt,inner sep=.5pt]
    %\node[color=red](v0) at (2,2){$v_0$};
    \node(v1) at (0,0){$\pi_1$};
    \node[](v2) at (1,1.5){$\pi_2$};
    \node[](v3) at (2,3){$\pi_3$};
    \node[](v4) at (3,4.5){$\pi_4$};
    \node[](v5) at (4,6){$\pi_5$};
    \node[](v6) at (5,7.5){$\pi_6$};
    \node[](v7) at (2,0){$\ell_1$};
    \node[](v8) at (3,1.5){$\ell_2$};
    \node[](v9) at (4,3){$\ell_3$};
    \node[](v10) at (5,4.5){$\ell_4$};
    \node[](v11) at(6,5.5){$\ell_5$};
    \draw(v1)--(v2)--(v3)--(v4);
    \draw (v4)--(v5)--(v6);
    %\draw(v0)--(v1);
    \draw(v2)--(v7);
    \draw(v3)--(v8);
    \draw(v4)--(v9);
    \draw(v5)--(v10);
    \draw(v6)--(v11);
  \tikzstyle{every node}=[]
\end{tikzpicture}
\end{center}

\begin{theorem}\label{comb-tree-theorem}
    For any $p \geq 2$ one has $\mu(\Cone(T_p)) = 1$, but $\ell(T_p)=p$.
\end{theorem}

\vspace{.5cm}

\section{Proof of Theorem~\ref{spanning-tree-number-theorem}} 
\label{tau section}

We recall here the statement of Theorem~\ref{spanning-tree-number-theorem}.

\vskip.1in
\noindent
\textbf{Theorem~\ref{spanning-tree-number-theorem}.}
{\it For $p \geq 1$ and $s_1,s_2 \geq 1$, the cone over the bi-coconut tree $T(p,s_1,s_2)$ has 
$$
\tau(\Cone(T(p,s_1,s_2)))=t(p,s_1,s_2)
:= 2^{s_2-1}(2b_{2p-3}^{(s_1)} +s_2 b_{2p-4}^{(s_1)}).
$$
}
\vskip.1in

\begin{proof}
  We proceed by induction on $s_1.$ When $s_1 = 1$ we have 
\begin{align}
    \tau(\Cone(T(p,1,s_2) ))
                            & = \tau(\Cone(\CT{p+1}{s_2})) \notag \\
                            \label{statement2}
                            & = 2^{s_2-1}(2\cdot F_{2p+2} +s_2(F_{2p+1})) \\   
                            & = 2^{s_2-1}(2\cdot b_{2p-3}^{(1)} +s_2b_{2p-4}^{(1)})\label{statement3}
\end{align}
as desired: equation \eqref{statement2} follows from  \eqref{eq:1} letting $s_2 = s$; and, equation \eqref{statement3} follows from the definition of $b_n^{(s_1)}$ described in \eqref{b-F-relation} when $s_1 = 1.$ 
Next we apply equation \eqref{thm 1.5-1} in Lemma~\ref{old-theorem-used-as-lemma}.  That is, we select, without loss of generality, any leaf in $T(p,s_1,s_2)$ in the appended left star and let its unique neighbor be $\pi_1,$ the first vertex in the path of $T(p,s_1,s_2)$, obtaining:
\begin{align}
   \tau(\Cone(T(p,s_1,s_2) )) 
   &= \tau(\Cone(T(p,s_1-1,s_2) ))+ \tau(Cone^{(+\pi_1)}(T(p,s_1-1,s_2) )) \label{spanning_tree_equality}.   
    \end{align}
     For example, the spanning tree number of the graph $\Cone(T(5,5,4))$  equals the sum of the spanning tree number of the  2 graphs below it: 
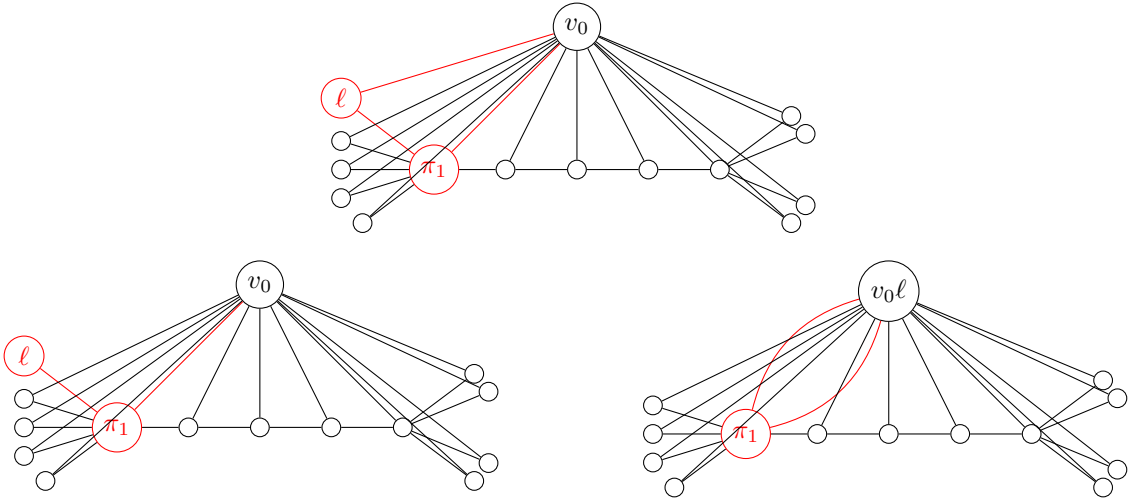
\vspace{.1cm}
   
\begin{center}
%first fig
\begin{figure}[H]
\begin{tikzpicture}[scale = 0.95]
 
    \tikzstyle{every node}=[draw,circle,minimum size=4pt,inner sep=2.5pt]
    \node[](v0'') at (-1.3,.4){};
    \node[](v0') at (-1.3,-.4){};
    \node[](v0) at (2,2){$v_0$};
    \node[color=red](v1) at (0,0){$\pi_1$};
    \node[](v2) at (1,0){};
    \node[](v3) at (2,0){};
    \node[](v4) at (3,0){};
    \node[](v5) at (4,0){};
    \node[](v6) at (5,.75){};
    \node[](v7) at (5,-.75){};
    \node[](v8) at (5.2,-.5){};
    \node[](v9) at (5.2,.5){};
     \node[color=red](v10) at(-1.3,1){$\ell$};
     \node[](v11) at (-1.3,0){};
     \node[](v12) at (-1.0,.-.75){};
    
    \draw(v1)--(v2)--(v3)--(v4);
    
    \draw (v4)--(v5)--(v6);
    \draw(v9)--(v5)--(v7);
    \draw(v5)--(v8);
   \draw(v0')--(v1);
   \draw(v0'')--(v1);
    \draw[color=red](v10)--(v1);
    \draw(v11)--(v1);
    \draw(v12)--(v1);
     \def\C{(5.75,1)} 
     \def\D{(5.75,-1)} 
     \def\E{(0,-.75)} 
     \def\F{(4,-.75)} 

     % connect cone
 
      \draw[color=red](v0)--(v1);
      \draw(v0)--(v0'');
      \draw(v0)--(v0');
      \draw(v0)--(v2);
      \draw(v0)--(v3);
      \draw(v0)--(v4);
      \draw(v0)--(v5);
      \draw(v0)--(v6);
      \draw(v0)--(v7);
      \draw(v0)--(v8);
      \draw(v0)--(v9);
      \draw[color=red](v0)--(v10);
      \draw(v0)--(v11);
      \draw(v0)--(v12);

  \tikzstyle{every node}=[]
    
  \coordinate (A) at (6,0);
   \coordinate (B) at (2,-1);
  
 \draw[] (A) circle (0);

  \end{tikzpicture}

 \end{figure}

%second fig
\begin{multicols}{2}
\begin{figure}[H]
\begin{tikzpicture}[scale = 0.95]
    \tikzstyle{every node}=[draw,circle,minimum size=4pt,inner sep=2.5pt]
    \node[](v0'') at (-1.3,.4){};
    \node[](v0') at (-1.3,-.4){};
    \node[](v0) at (2,2){$v_0$};
    \node[color=red](v1) at (0,0){$\pi_1$};
    \node[](v2) at (1,0){};
    \node[](v3) at (2,0){};
    \node[](v4) at (3,0){};
    \node[](v5) at (4,0){};
    \node[](v6) at (5,.75){};
    \node[](v7) at (5,-.75){};
    \node[](v8) at (5.2,-.5){};
    \node[](v9) at (5.2,.5){};
    \node[color=red](v10) at(-1.3,1){$\ell$};
     \node[](v11) at (-1.3,0){};
     \node[](v12) at (-1.0,.-.75){};
    
    \draw(v1)--(v2)--(v3)--(v4);
    
    \draw (v4)--(v5)--(v6);
    \draw(v9)--(v5)--(v7);
    \draw(v5)--(v8);
   \draw(v0')--(v1);
   \draw(v0'')--(v1);
    \draw[color=red](v10)--(v1);
    \draw(v11)--(v1);
    \draw(v12)--(v1);
     \def\C{(5.75,1)} 
     \def\D{(5.75,-1)} 
     \def\E{(0,-.75)} 
     \def\F{(4,-.75)} 

     % connect cone
      \draw[color=red](v0)--(v1);
      \draw(v0)--(v0'');
      \draw(v0)--(v0');
      \draw(v0)--(v2);
      \draw(v0)--(v3);
      \draw(v0)--(v4);
      \draw(v0)--(v5);
      \draw(v0)--(v6);
      \draw(v0)--(v7);
      \draw(v0)--(v8);
      \draw(v0)--(v9);
      
      \draw(v0)--(v11);
      \draw(v0)--(v12);

  \tikzstyle{every node}=[]
    
  \coordinate (A) at (6,0);
   \coordinate (B) at (2,-1);
  
 \draw[] (A) circle (0);
 
\end{tikzpicture}
\end{figure}

%3rd figure
\begin{figure}[H]
\begin{tikzpicture}[scale = 0.95]
   
    \tikzstyle{every node}=[draw,circle,minimum size=4pt,inner sep=2.5pt]
    \node[](v0'') at (-1.3,.4){};
    \node[](v0') at (-1.3,-.4){};
    \node(v0) at (2,2){$v_0\ell$};
    \node[color=red](v1) at (0,0){$\pi_1$};
    \node[](v2) at (1,0){};
    \node[](v3) at (2,0){};
    \node[](v4) at (3,0){};
    \node[](v5) at (4,0){};
    \node[](v6) at (5,.75){};
    \node[](v7) at (5,-.75){};
    \node[](v8) at (5.2,-.5){};
    \node[](v9) at (5.2,.5){};
     %\node[color=red](v10) at(-1.0,.75){$\ell$};
     \node[](v11) at (-1.3,0){};
     \node[](v12) at (-1.0,.-.75){};
    
    \draw(v1)--(v2)--(v3)--(v4);
    
    \draw (v4)--(v5)--(v6);
    \draw(v9)--(v5)--(v7);
    \draw(v5)--(v8);
   \draw(v0')--(v1);
   \draw(v0'')--(v1);
    %\draw[color=red](v10)--(v1);
     \draw [color=red][bend left] (v1) to (v0);
    \draw [color=red][bend right] (v1) to (v0);
    \draw(v11)--(v1);
    \draw(v12)--(v1);
     \def\C{(5.75,1)} 
     \def\D{(5.75,-1)} 
     \def\E{(0,-.75)} 
     \def\F{(4,-.75)} 

     % connect cone
      %\draw[color=red](v0)--(v1);
      \draw(v0)--(v0'');
      \draw (v0)--(v0');
      \draw(v0)--(v2);
      \draw(v0)--(v3);
      \draw(v0)--(v4);
      \draw(v0)--(v5);
      \draw(v0)--(v6);
      \draw(v0)--(v7);
      \draw(v0)--(v8);
      \draw(v0)--(v9);
      \draw(v0)--(v11);
      \draw(v0)--(v12);

  \tikzstyle{every node}=[]
    
  \coordinate (A) at (6,0);
   \coordinate (B) at (2,-1);
  
 \draw[] (A) circle (0);

  \end{tikzpicture}
  
  \end{figure}
  \end{multicols}
\end{center}

\noindent Now delete and contract one of the double edges of 
$\Cone^{(+\pi_1)}(T(p,s_1-1,s_2) )$, as pictured for $\Cone^{(+\pi_1)}(T(5,4,4) )$:
\vspace{.1cm}
     %3rd figure repeated
\begin{center}

\begin{figure}[H]
\begin{tikzpicture}
  
    \tikzstyle{every node}=[draw,circle,minimum size=4pt,inner sep=2.5pt]
    \node[](v0'') at (-1.3,.4){};
    \node[](v0') at (-1.3,-.4){};
    \node(v0) at (2,2){$v_0\ell$};
    \node[color=red](v1) at (0,0){$\pi_1$};
    \node[](v2) at (1,0){};
    \node[](v3) at (2,0){};
    \node[](v4) at (3,0){};
    \node[](v5) at (4,0){};
    \node[](v6) at (5,.75){};
    \node[](v7) at (5,-.75){};
    \node[](v8) at (5.2,-.5){};
    \node[](v9) at (5.2,.5){};
     \node[](v11) at (-1.3,0){};
     \node[](v12) at (-1.0,.-.75){};
    
    \draw(v1)--(v2)--(v3)--(v4);
    
    \draw (v4)--(v5)--(v6);
    \draw(v9)--(v5)--(v7);
    \draw(v5)--(v8);
   \draw(v0')--(v1);
   \draw(v0'')--(v1);
     \draw [color=red][bend left] (v1) to (v0);
    \draw [color=red][bend right] (v1) to (v0);
    \draw(v11)--(v1);
    \draw(v12)--(v1);
     \def\C{(5.75,1)} 
     \def\D{(5.75,-1)} 
     \def\E{(0,-.75)} 
     \def\F{(4,-.75)} 

     % connect cone
    
      \draw(v0)--(v0'');
      \draw (v0)--(v0');
      \draw(v0)--(v2);
      \draw(v0)--(v3);
      \draw(v0)--(v4);
      \draw(v0)--(v5);
      \draw(v0)--(v6);
      \draw(v0)--(v7);
      \draw(v0)--(v8);
      \draw(v0)--(v9);
      \draw(v0)--(v11);
      \draw(v0)--(v12);

  \tikzstyle{every node}=[]
    
  \coordinate (A) at (6,0);
   \coordinate (B) at (2,-1);
  
 \draw[] (A) circle (0);

  \end{tikzpicture}
 
  \end{figure}
\end{center}
\newpage

\begin{multicols}{2}

\begin{figure}[H]
\begin{tikzpicture}
   
    \tikzstyle{every node}=[draw,circle,minimum size=4pt,inner sep=2.5pt]
    \node[](v0'') at (-1.3,.4){};
    \node[](v0') at (-1.3,-.4){};
    \node(v0) at (2,2){$v_0\ell$};
    \node[color=red](v1) at (0,0){$\pi_1$};
    \node[](v2) at (1,0){$\pi_2$};
    \node[](v3) at (2,0){};
    \node[](v4) at (3,0){};
    \node[](v5) at (4,0){};
    \node[](v6) at (5,.75){};
    \node[](v7) at (5,-.75){};
    \node[](v8) at (5.2,-.5){};
    \node[](v9) at (5.2,.5){};
     
     \node[](v11) at (-1.3,0){};
     \node[](v12) at (-1.0,.-.75){};
    
    \draw(v1)--(v2)--(v3)--(v4);
    
    \draw (v4)--(v5)--(v6);
    \draw(v9)--(v5)--(v7);
    \draw(v5)--(v8);
   \draw(v0')--(v1);
   \draw(v0'')--(v1);
   
    \draw [color=red][bend right] (v1) to (v0);
    \draw(v11)--(v1);
    \draw(v12)--(v1);
     \def\C{(5.75,1)} 
     \def\D{(5.75,-1)} 
     \def\E{(0,-.75)} 
     \def\F{(4,-.75)} 

     % connect cone
      %\draw[color=red](v0)--(v1);
      \draw(v0)--(v0'');
      \draw (v0)--(v0');
      \draw(v0)--(v2);
      \draw(v0)--(v3);
      \draw(v0)--(v4);
      \draw(v0)--(v5);
      \draw(v0)--(v6);
      \draw(v0)--(v7);
      \draw(v0)--(v8);
      \draw(v0)--(v9);
      %\draw[color=red](v0)--(v10);
      \draw(v0)--(v11);
      \draw(v0)--(v12);

  \tikzstyle{every node}=[]
    
  \coordinate (A) at (6,0);
   \coordinate (B) at (2,-1);
  
 \draw[] (A) circle (0);

  \end{tikzpicture}

\end{figure}

\begin{figure}[H]
\begin{tikzpicture}
  
    \tikzstyle{every node}=[draw,circle,minimum size=4pt,inner sep=2.5pt]
    \node[](v0'') at (-1.3,1.4){};
    \node[](v0') at (-1.3,-.2){};
    \node(v0) at (2,2){$v_0\ell\pi_1$};
    
    \node[](v2) at (1,0){$\pi_2$};
    \node[](v3) at (2,0){};
    \node[](v4) at (3,0){};
    \node[](v5) at (4,0){};
    \node[](v6) at (5,.75){};
    \node[](v7) at (5,-.75){};
    \node[](v8) at (5.2,-.5){};
    \node[](v9) at (5.2,.5){};
     \node[](v11) at (-1.3,1){};
     \node[](v12) at (-1.3,.65){};
    
    \draw(v2)--(v3)--(v4);
    
    \draw (v4)--(v5)--(v6);
    \draw(v9)--(v5)--(v7);
    \draw(v5)--(v8);
     \def\C{(5.75,1)} 
     \def\D{(5.75,-1)} 
     \def\E{(0,-.75)} 
     \def\F{(4,-.75)} 

     % connect cone
      %\draw[color=red](v0)--(v1);
      \draw[bend right](v0) to (v0'');
      \draw[color=red][bend left](v0) to (v0'');

      \draw[bend right](v0) to (v0');
      \draw[color=red][bend left](v0) to (v0');
      
     \draw[bend right](v0) to (v11);
      \draw[color=red][bend left](v0) to (v11);

      \draw[bend right](v0) to (v12);
      \draw[color=red][bend left](v0) to (v12);
      \draw(v0)--(v2);
      \draw(v0)--(v3);
      \draw(v0)--(v4);
      \draw(v0)--(v5);
      \draw(v0)--(v6);
      \draw(v0)--(v7);
      \draw(v0)--(v8);
      \draw(v0)--(v9);

  \tikzstyle{every node}=[]
    
  \coordinate (A) at (6,0);
   \coordinate (B) at (2,-1);
  
 \draw[] (A) circle (0);

\draw [color=red] (v0) to [loop right] ();
  \end{tikzpicture}
\end{figure}
\end{multicols}

     Then equation \eqref{spanning_tree_equality} becomes
    
    \begin{align}
     \tau(\Cone(T(p,s_1,s_2) ))
    &= 2\tau(\Cone(T(p,s_1-1,s_2) ))+ 2^{s_1-1}\tau (\Coneplus{\pi_1}(\CT{p-1}{s_2}))\\
\hspace{4.8cm} &= 2\tau(\Cone(T(p,s_1-1,s_2) ))+ 2^{s_1-1}(2^{s_2-1}(2F_{2p}+(s_2-2)F_{2p-2}))\label{induct 2.7-2}\\
   &= 2\cdot 2^{s_2-1}(2\cdot b_{2p-3}^{(s_1-1)} +s_2b_{2p-4}^{(s_1-1)})+2^{s_1+s_2-2}(2F_{2p}+(s_2-2)F_{2p-2})\label{induct 2.7-3}.
   \end{align}
  Equation \eqref{induct 2.7-2} follows from  equation \eqref{eq:2}; and, \eqref{induct 2.7-3} follows from substituting the value from equation \eqref{eq:1}. We now hope to show that the right side of equation \eqref{induct 2.7-3}  equals $2^{s_2-1}(2b_{2p-3}^{(s_1)} +s_2b_{2p-4}^{(s_1)})$  as in the right side of Theorem~\ref{spanning-tree-number-theorem}. That is, our goal is to prove \begin{align}\label{spanning-simultateneous-induction}
2b_{2p-3}^{(s_1)} +s_2b_{2p-4}^{(s_1)} &= 2(2b_{2p-3}^{(s_1-1)} +s_2b_{2p-4}^{(s_1-1)})+2^{s_1-1}(2F_{2p} +(s_2-2)F_{2p-2})
\end{align}
Equation \eqref{spanning-simultateneous-induction} would be implied by these two stronger assertions.
\begin{align}
b_{2p-4}^{(s_1)}&=  2b_{2p-4}^{(s_1-1)} + 2^{s_1-1}(F_{2p-2})\label{b_n neven}\\ 
b_{2p-3}^{(s_1)}&=  2b_{2p-3}^{(s_1-1)} + 2^{s_1-1}(F_{2p}-F_{2p-2})
\label{b_n n odd}\end{align}
We now prove \eqref{b_n neven} and \eqref{b_n n odd} by  simultaneous induction on $p$.  

For the sake of checking base cases, when $p=1,2$, one can check that  \eqref{b_n neven} and \eqref{b_n n odd}  respectively become tautologies, and 
one can also directly check that \eqref{b_n neven} and \eqref{b_n n odd} hold for $p= 3,4.$

In the inductive step, letting $p=k+1$, then  \eqref{b_n neven}  yields
\begin{align*}
b_{2p-4}^{(s_1)} = b_{2k-2}^{(s_1)}:&=  b_{2k-3}^{(s_1)} + b_{2k-4}^{(s_1)} \\ 
& =  2b_{2k-3}^{(s_1-1)} + 2^{s_1-1}(F_{2k}-F_{2k-2})+2b_{2k-4}^{(s_1-1)} + 2^{s_1-1}(F_{2k-2})\\
& =  2b_{2k-2}^{(s_1-1)} + 2^{s_1-1}F_{2k}\\ 
& =  2b_{2p-4}^{(s_1-1)} + 2^{s_1-1}F_{2p-2} 
\end{align*}
as desired. Similarly,  \eqref{b_n n odd} yields
\begin{align*}
b_{2p-3}^{(s_1)} = b_{2k-1}^{(s_1)}:&=  b_{2k-2}^{(s_1)} + b_{2k-3}^{(s_1)} \\ 
& =  2b_{2k-2}^{(s_1-1)} + 2^{s_1-1}(F_{2k})+b_{2k-3}^{(s_1-1)} + 2^{s_1-1}(F_{2k}-F_{2k-2})\\
& = 2b_{2k-2}^{(s_1-1)} + 2b_{2k-3}^{(s_1-1)}+2^{s_1-1}(F_{2k}+F_{2k}-F_{2k-2}).
\end{align*}
\vspace{1cm}
It remains to show 
    \hspace{1.4cm}$2F_{2k}-F_{2k-2} = F_{2k+2} -F_{2k}.$
    
\noindent
Observe the following:
\begin{align*}
F_{2k+1} &= F_{2k} + F_{2k-1} +F_{2k-2} - F_{2k-2}\\
F_{2k+1} +F_{2k} &= 3F_{2k} - F_{2k-2}\\
F_{2k+2} -F_{2k} &= 2F_{2k} - F_{2k-2}
\end{align*} as desired. Thus 
\begin{align*}
b_{2k-1}^{(s_1)}&= 2b_{2k-1}^{(s_1-1)}+2^{s_1-1}(F_{2k+2} -F_{2k})\\ 
\end{align*} 
and hence
\begin{align*} 
b_{2p-3}^{(s_1)}&=  2b_{2p-3}^{(s_1-1)} + 2^{s_1-1}(F_{2p}-F_{2p-2})
\end{align*} 
This completes the proof of Theorem~\ref{spanning-tree-number-theorem}.
\end{proof}

%%%%%%%%%%%%%%%%%%
\section{Proof of Theorem~\ref{generating-function-theorem}}
\label{gen function}

We recall the statement of the theorem.
\vskip.1in
\noindent
{\bf Theorem~\ref{generating-function-theorem}.}
{\it 
For $s_1,s_2 \geq 1$,  $t(p,s_1,s_2)$ has
ordinary generating function in $p$ given by
$$
\displaystyle\sum_{p=1}^\infty t(p, s_1, s_2)x^{p-1}
=2^{s_1+s_2-1} \cdot \displaystyle\frac{4 + 2(s_1+s_2)(1-x) +s_1s_2x}{1-3x+x^2}.
$$ 
In particular, this function is symmetric in $s_1,s_2$, so the same holds for $t(p,s_1,s_2)$.
}

\begin{proof}
Fix $s_1$ and define $b_n$ as in equations $\eqref{b_n},\eqref{b_-2}, \eqref{b_-1}\footnote{There are no superscripts on $b_n;$ it should be assumed that all are functions of $s_1$.}.$ Then

\begin{align}
    B(x) &:= b_{-2} + b_{-1}x+b_0x^2+\cdots\label{series}
        = \displaystyle\sum^\infty_{n=-2}b_nx^{n+2}\\
        &= b_{-2} + b_{-1}x+\displaystyle\sum^\infty_{n=0}b_nx^{n+2} \notag \\
        &= b_{-2} + b_{-1}x+\displaystyle\sum^\infty_{n=0}b_{n-1}x^{n+2} + \displaystyle\sum^\infty_{n=0}b_{n-2}x^{n+2}\notag\\
        &= b_{-2} + b_{-1}x+x\displaystyle\sum^\infty_{n=0}b_{n-1}x^{n+1} + x^2\displaystyle\sum^\infty_{n=0}b_{n-2}x^{n} \notag
\end{align}

Hence equation \eqref{series} becomes 
\begin{align}
    B(x) &= b_{-2}+b_{-1}x+x(B(x)-b_{-2})+ x^2B(x) \label{eq:functional-equation1}\\
     B(x)-xB(x)-x^2B(x) &= b_{-2}+b_{-1}x+-b_{-2}x\label{eq:functional-equation2}\\
    B(x)&= \displaystyle\frac{b_{-2}+(b_{-1}-b_{-2})x}{1-x-x^2}\label{eq:functional-equation-solved}\\
    &= \displaystyle\frac{2^{s_1-1}(2+s_1x)}{1-x-x^2}\label{closed_form}
\end{align}
where this last equation
\eqref{closed_form} results from using the initial values  \eqref{b_-2}, \eqref{b_-1}.

Next recall from Definition~\ref{t-definition} that  
\begin{equation*}
t(p,s_1,s_2):=2^{s_2-1}(2\cdot b_{2p-3} +s_2b_{2p-4})
\end{equation*} for $p \geq 1$ and $s_1,s_2\geq 1$. Define its generating function in $p$ as
\begin{align}
    T(x) &:= \displaystyle\sum_{p=1}^\infty t(p,s_1,s_2)x^{p-1}\notag\\
      &=\displaystyle\sum_{p=1}^\infty 2^{s_2-1}(2\cdot b_{2p-3} +s_2b_{2p-4})x^{p-1}\notag\\
       &=2^{s_2-1}\left(2\displaystyle\sum_{p=1}^\infty  (b_{2p-3} )x^{p-1} + s_2\displaystyle\sum_{p=1}^\infty b_{2p-4}x^{p-1}\right)\notag\\
       &=2^{s_2-1}\left(2(b_{-1} + b_1x + b_3x^2 +b_5x^3+\cdots) + s_2(b_{-2} + b_0x + b_2x^2 +b_4x^3+\cdots)\right)\label{sumof2series}\\
       &= 2^{s_2-1}\left[2\cdot\displaystyle\frac{1}{2}\left[\displaystyle\frac{B(x)-B(-x)}{x} \right]_{x \mapsto x^{1/2}} + s_2\cdot\displaystyle\frac{1}{2}
       \left[\displaystyle\frac{B(x)+B(-x)}{x}\right]_{x \mapsto x^{1/2}} \right] \label{multi-series} \\
       &= 2^{s_1+s_2-2}\left[\displaystyle\frac{4+2(s_1+s_2)(1-x)+s_1s_2x}{1-3x+x^2}\right].
       \label{symmetry}
\end{align}
Here equation \eqref{multi-series} was obtained  by taking the bisection of the series in \eqref{sumof2series} and substituting $x^{1/2}$ for $x.$ Equation \eqref{symmetry} was then obtained by substituting \eqref{closed_form} for $B(x)$ in \eqref{multi-series}, followed by  some further algebraic simplication.  Note that the result in \eqref{symmetry} is symmetric in $s_1$ and $s_2.$
\end{proof}

\section{Proof of Theorem~\ref{sandpile-group-theorem}}\label{group structure}

Recall that for positive integers $p,s_1,s_2$,
the bi-coconut tree $T(p,s_1,s_2)$
is obtained from a path having $p$ vertices $\pi_1,\pi_2,\ldots,\pi_p$ by adding $s_1$ leaves attached to the
endvertex $\pi_1$ and $s_2$ leaves attached to
the end vertex $\pi_p$.  We label the leaves
as illustrated in the diagram of $T(6,5,4)$ shown below:
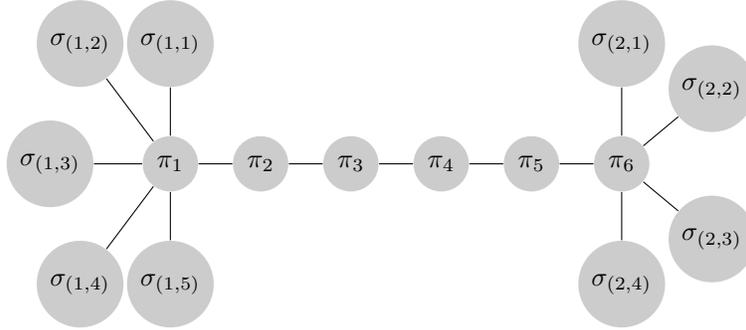
\begin{center}
  \input{drawing3}
    \label{fig:labeled bi-coconuttree}
\end{center}

\medskip
Recall the statement of the theorem.

\vskip.1in
\noindent \textbf{Theorem~\ref{sandpile-group-theorem}}
\textit{For $p \geq 2$
and $s_1,s_2 \geq 1$, the cone over the bi-coconut tree $T(p,s_1,s_2)$ 
has the following sandpile group structure:}
\begin{align}
    K(\Cone(T(p,s_1,s_2) )) 
                             \cong                                
\begin{cases}
    \Z_2^{s_1+s_2-2} \oplus \Z_a 
    & \text{ if } p \equiv 1 \bmod{3},\\
    \Z_2^{s_1+s_2-3} \oplus \Z_{2a}   
    & \text{ if }s_1   \text{ or } s_2 \text{ odd } \text{ and }p \not\equiv 1 \bmod{3},\\
    \Z_2^{s_1+s_2-4} \oplus \Z_4\oplus\Z_{a}  
    & \text{ if }s_1 \text{ and } s_2  \text{ even  and }p \not\equiv 1 \bmod{3}.
\end{cases}\label{main thm}\notag
\end{align}
\textit{Here $a$ is defined as in \eqref{formula-for-a} via this formula
\begin{equation*}
%\label{formula-for-a}
    a =2^{1-s_1}(2\cdot b_{2p-3} +s_2b_{2p-4})
    =\frac{t(p,s_1,s_2)}{2^{s_1+s_2-2}},
\end{equation*}
and $b_n = b_{n-1} + b_{n-2}$ with $b_{-2} = 2^{s_1}, 
b_{-1} = 2^{s_1-1}(s_1+2)$, was defined\footnote{We have omitted the superscripts on $b_n=b_n^{(s_1)}$ here, but it should be assumed that all are functions of $s_1$.}
in \eqref{b_n},\eqref{b_-2},\eqref{b_-1}}.

\subsection{Proof outline}
 
The proof of Theorem~\ref{sandpile-group-theorem} proceeds in several steps, outlined here.

Recall that $v_0$ denotes the cone vertex in the graph $G=\Cone(T)$ on vertex set $V$.  Let $\{e_v\}$ be the standard basis vectors
for $\Z^{V - \{v_0\}}$, and denote by
$\{ \overline{e}_v \}$ their images within
the quotient presentation
of the sandpile group 
$K(G)=\Z^{V-\{v_0\}}/\im \overline{L}_G$
as described in equation \eqref{sandpile-group-presentation}. 

As mentioned in the Introduction, a crucial role is played by the following result on generators for $K(G)$ when $G=\Cone(T)$ for any tree $T$.

\vskip.1in
\noindent
{\bf Theorem~\ref{generators}.} \cite{Reiner2024SandpileGF} 
{\it Choosing the cone vertex of $\Cone(T)$ as the root 
vertex $v_0$, the group $K(\Cone(T))$
is generated by 
$
\{\overline{e}_v: \text{ leaves }v \neq v_1\}
$
for any choice of a leaf vertex $v_1$ in $T$.
}
\vskip.1in

In particular, when $T=T(p,s_,s_2)$, we label the $s_1$ leaf vertices as $\sigma_{(1,1)},\sigma_{(1,2)},\ldots,\sigma_{(1,s_1)},$  the  $p$ vertices in the path as $\pi_1,\pi_2,\ldots,\pi_p$ and the  $s_2$ leaf vertices as $\sigma_{(2,1)},\sigma_{(2,2)},\ldots,\sigma_{(2,s_2)}$ as in Figure \ref{fig:labeled bi-coconuttree}.  

Applying Theorem~\ref{generators},
one concludes that the images of the basis
elements indexed by all of the leaves
\begin{equation}
\label{leaf-images}
\bar{e}_{\sigma_{(1,1)}},
\bar{e}_{\sigma_{(1,2)}},
\ldots
\bar{e}_{\sigma_{(1,s_1)}},
\bar{e}_{\sigma_{(2,1)}},
\ldots
\bar{e}_{\sigma_{(2,s_2)}},
\end{equation}
will generate the sandpile group $K:=K(\Cone(T(p,s_1,s_2))$.

\begin{itemize}
\item [{\sf Step 1.}]
We start by looking for relations in $K$ that hold among the  elements in \eqref{leaf-images}, first by looking for relations that also involve the
extra two elements 
$
\bar{e}_{\pi_{p-1}},
\bar{e}_{\pi_p}
$
and later eliminating them.
These relations will be encoded as the columns of a
$(2+s_1+s_2)\times (2+s_1+s_2)$ matrix $M$, shown in 
\eqref{M-matrix} and \eqref{M-abbreviated} below,
whose rows are indexed by 
\begin{equation}
\label{leaves-plus-path-endpoints}
\bar{e}_{\pi_{p-1}},
\bar{e}_{\pi_p},\bar{e}_{\sigma_{(1,1)}},
\bar{e}_{\sigma_{(1,2)}},
\ldots
\bar{e}_{\sigma_{(1,s_1)}},
\bar{e}_{\sigma_{(2,1)}},
\ldots
\bar{e}_{\sigma_{(2,s_2)}}.
\end{equation}

\item [{\sf Step 2.}]
By performing column operations on the matrix $M$, we produce
an $(s_1+s_2) \times (s_1+s_2)$ matrix $M'$, shown in \eqref{starting_matrix} below, whose columns encode relations
among only the elements in \eqref{leaf-images}.

\item [{\sf Step 3.}]
We show that $\pm \det(M')=|K|=\tau(\Cone(T(p,s_1,s_2))=t(p,s_1,s_2)$.
This is achieved 
via invertible row and column operations over $\mathbb{Z}$
that only affect $\det(M')$ up to sign, bringing it to an lower triangular form
$M''$, shown in \eqref{block-upper-triangular-matrix} below, and then checking that $\pm \det(M'')=t(p,s_1,s_2)$.
\item [{\sf Step 4.}]
We conclude $K \cong \Z^{s_1 + s_2}/\text{im}(M')$ via the following reasoning. Since Theorem~\ref{generators} implied that 
$K$ is generated by the $s_1+s_2$ elements
in \eqref{leaf-images},
this gives a surjection
$
\pi: \Z^{s_1 + s_2} \twoheadrightarrow K.
$
By construction, the columns of $M'$ are relations in $K$ satisfied by the elements of \eqref{leaf-images}, so $\pi$ 
descends to a well-defined surjection
$
\bar{\pi}: \Z^{s_1 + s_2}/\text{im}(M') \twoheadrightarrow K.
$
Then by {\sf Step 3},
this surjection $\bar{\pi}$ must be an isomorphism, since the source and target would have the same cardinality:
$
|\Z^{s_1 + s_2}/\text{im}(M')|=|\det(M')|=|K|.
$

\item [{\sf Step 5.}] We compute the Smith normal form of $M'$ in the three cases from Theorem~\ref{sandpile-group-theorem}, depending upon $p \bmod{3}$ and $s_1, s_2 \bmod{2}$,
thereby determining $K \cong \Z^{s_1+s_2}/\im(M')$.
\end{itemize}

\subsection{Step 1.}

We look for relations in $K$ among the elements in \eqref{leaf-images}, by first looking for relations that also involve the
extra two elements 
$
\bar{e}_{\pi_{p-1}},
\bar{e}_{\pi_p}
$
appearing in \eqref{leaves-plus-path-endpoints},
which will be eliminated in {\sf Step 2}.
Recall that the Laplacian, $L_G,$ is defined as the diagonal degree matrix minus the adjacency matrix of $G$, and its reduced
Laplacian $\overline{L}_G$ removes the row and column corresponding to some chosen vertex $v_0$. The columns of the reduced Laplacian
give us the defining relations for $K=K(G)=\Z^{V-v_0}/\im(\overline{L}_G)$.

For $G=\Cone(T(p,s_1,s_2))$, we choose $v_0$ to be the cone vertex. For columns corresponding to leaf vertices, we have
these relations in $K$:
\begin{align}
2\bar{e}_{\sigma_{(1,i)}} = \bar{e}_{\pi_1} \text{ for }i = 1, \dots, s_1,\label{sigma1_pi_1}\\
2\bar{e}_{\sigma_{(2,j)}} = \bar{e}_{\pi_p} \text{ for }j = 1, \dots, s_2\label{sigma2_pi_p}.  
\end{align}

For the columns corresponding to vertices $\pi_1$ and $\pi_p$ we have 
\begin{align}(s_1+2)\bar{e}_{\pi_1} &= \bar{e}_{\sigma_{(1,1)}}+
\bar{e}_{\sigma_{(1,2)}}+
\ldots
+\bar{e}_{\sigma_{(1,s_1)}}+ \bar{e}_{\pi_2}\label{column_1}\\
(s_2+2)\bar{e}_{\pi_p} &= \bar{e}_{\sigma_{(2,1)}}+
\bar{e}_{\sigma_{(2,2)}}+
\ldots
+\bar{e}_{\sigma_{(2,s_2)}}+ \bar{e}_{\pi_{p-1}}\label{column_2}
\end{align}
The columns that correspond to other vertices in the path 
($\pi_i$ for $i = 2,\ldots, p-1$ ), one has these relations:
\begin{align}
 3\bar{e}_{\pi_i}  =  \bar{e}_{\pi_{i-1}} + \bar{e}_{\pi_{i+1}}.\label{path_relation}
\end{align}

\begin{lemma}
\label{CT-trunk-lemma}
For $i=2,\ldots,p$ one has in  $K=K(\Cone(T(p,s_1,s_2)))$ that
% p-1 since this must correspond to vertices not adjacent to star
\begin{align}   
\bar{e}_{\pi_1}
&= 
(F_{2i-2} )\bar{e}_{\pi_{i-1}}-(F_{2i-4})\bar{e}_{\pi_{i}} \label{e_1}\\
\bar{e}_{\pi_2}
&= 
(F_{2i-4} )\bar{e}_{\pi_{i-1}}-(F_{2i-6}) \bar{e}_{\pi_{i}}\label{e_2} . 
\end{align}

\end{lemma}

\begin{proof} We begin with equation \eqref{e_1}.
Induct on $i$.  The base case $i=2$ asserts 
\begin{equation*}
\bar{e}_{\pi_1} =  \bar{e}_{\pi_1}\\
\end{equation*}
The base case $i=3$ asserts 

\begin{equation*}
\bar{e}_{\pi_1}=  3\bar{e}_{\pi_2} - \bar{e}_{\pi_3}
\end{equation*} which is true by equation \eqref{path_relation}.
In the inductive step, assuming equation \eqref{e_1} at  $i-1$,
which asserts 
$$
\bar{e}_{\pi_1}=(F_{2i-4} )\bar{e}_{\pi_{i-2}}-(F_{2i-6}) \bar{e}_{\pi_{i-1}},
$$
we wish to show the same assertion at $i$, which asserts
$$
\bar{e}_{\pi_1}=(F_{2i-2} )\bar{e}_{\pi_{i-1}}-(F_{2i-4})\bar{e}_{\pi_{i}}. 
$$
It therefore suffices to check this equality
$$ (F_{2i-4} )\bar{e}_{\pi_{i-2}}-(F_{2i-6}) \bar{e}_{\pi_{i-1}} \overset{?}{=}
(F_{2i-2} )\bar{e}_{\pi_{i-1}}-(F_{2i-4})\bar{e}_{\pi_{i}}
$$ 
which we verify step-by-step with justification below:
\begin{align}
\notag (F_{2i-4} )\bar{e}_{\pi_{i-2}}-(F_{2i-6}) \bar{e}_{\pi_{i-1}}  &\overset{?}{=}(F_{2i-2} )\bar{e}_{\pi_{i-1}}-(F_{2i-4})\bar{e}_{\pi_{i}} \\
\label{used-path-relation}&= (F_{2i-2} )\bar{e}_{\pi_{i-1}}-(F_{2i-4})(3\bar{e}_{\pi_{i-1}}-\bar{e}_{\pi_{i-2}}) \\
\notag-(F_{2i-6}) \bar{e}_{\pi_{i-1}} &=(F_{2i-2} )\bar{e}_{\pi_{i-1}}-(3F_{2i-4})\bar{e}_{\pi_{i-1}}\\
\notag 0&=
(F_{2i-2}-3F_{2i-4}+F_{2i-6})\bar{e}_{\pi_{i-1}}
\end{align}
where \eqref{used-path-relation} used \eqref{path_relation}. It therefore suffices to check that
$$F_{2i-2}-3F_{2i-4}+F_{2i-6}=0.$$
This is seen as follows:
\begin{align*}
F_{2i-2}-3F_{2i-4}+F_{2i-6}
&=F_{2i-2}-(F_{2i-4}+F_{2i-4}+F_{2i-4})+F_{2i-6}\\
&=F_{2i-2}-(F_{2i-4}+F_{2i-4}+F_{2i-5}+F_{2i-6})+F_{2i-6}\\
&=F_{2i-2}-(F_{2i-4}+F_{2i-4}+F_{2i-5})\\
&=F_{2i-2}-(F_{2i-4}+F_{2i-3})\\
&=F_{2i-2}-F_{2i-2}=0.
\end{align*}
The proof of \eqref{e_2} is similar.
\end{proof}

We use the  relations described above to 
write down a $(2+s_1+s_2) \times(2+s_1+s_2)$ matrix $M$, each of whose columns expresses the coefficients of a relation valid in the sandpile group $K$ among
\begin{equation}
\label{eq:s1-plus-s2-plus-2-elements}
\bar{e}_{\pi_{p-1}},
\bar{e}_{\pi_p},
\bar{e}_{\sigma_{(1,1)}},
\bar{e}_{\sigma_{(1,2)}},
\ldots
\bar{e}_{\sigma_{(1,s_1)}},
\bar{e}_{\sigma_{(2,1)}},
\ldots
\bar{e}_{\sigma_{(2,s_2)}},
\end{equation}
which index its rows. 

Note that Lemma \ref{CT-trunk-lemma} allows one to write  $\bar{e}_{\pi_{1}}$ and $\bar{e}_{\pi_{2}}$ in terms of $\bar{e}_{\pi_{p-1}}$ and $\bar{e}_{\pi_{p}}$:
evaluating equations \eqref{e_1} and \eqref{e_2} at $i=p$
results in 
\begin{align}   
\bar{e}_{\pi_1}
&= 
F_{2p-2} \bar{e}_{\pi_{p-1}}-F_{2p-4}\bar{e}_{\pi_{p}} \label{e_{1,p}}\\
\bar{e}_{\pi_2}
&= 
F_{2p-4} \bar{e}_{\pi_{p-1}}-F_{2p-6}\bar{e}_{\pi_{p}}\label{e_{2,p}} . 
\end{align}
Substituting these two expressions for 
$
\bar{e}_{\pi_1},
\bar{e}_{\pi_2}
$
from
\eqref{e_{1,p}}, \eqref{e_{2,p}}
into \eqref{column_1}, we obtain
$$
 (s_1+2)[F_{2p-2} \bar{e}_{\pi_{p-1}}-F_{2p-4}\bar{e}_{\pi_{p}} )]
 = \bar{e}_{\sigma_{(1,1)}}+
\ldots
+\bar{e}_{\sigma_{(1,s_1)}}+ [F_{2p-4} \bar{e}_{\pi_{p-1}}-F_{2p-6} \bar{e}_{\pi_{p}}]
$$
or equivalently,
\begin{equation}
[(s_1+2) F_{2p-2} - F_{2p-4}]\bar{e}_{\pi_{p-1}} +  [F_{2p-6} - (s_1+2)F_{2p-4}]\bar{e}_{\pi_{p}} 
= -\bar{e}_{\sigma_{(1,1)}}
-\ldots
-\bar{e}_{\sigma_{(1,s_1)}} \label{column1_final}.
\end{equation}

Also, substituting equation \eqref{e_{1,p}} into \eqref{sigma1_pi_1} yields

\begin{align}
2\bar{e}_{\sigma_{(1,i)}} &= F_{2p-2} \bar{e}_{\pi_{p-1}}-F_{2p-4}\bar{e}_{\pi_{p}} \label{e_{s_1+2,p}}\text{ for }i = 1, \dots, s_1.
\end{align}

Each of the columns of the matrix $M$ below expresses the coefficients of a relation valid in the sandpile group $K$ among the elements of \eqref{eq:s1-plus-s2-plus-2-elements} as follows:
\begin{itemize}

    \item equation \eqref{column1_final}: column 1 
    \item equation \eqref{column_2}: column 2
    \item  equation \eqref{e_{s_1+2,p}}: column 3 through  column $s_1+2$
    \item equation \eqref{sigma2_pi_p}: column $s_1+3$ through column $s_1+s_2+2$. 
\end{itemize}

\begin{equation}
\label{M-matrix}
M =
\bordermatrix{~ &  &  &  &  & &  \cr
\bar{e}_{\pi_{p-1}}&(s_1+2) F_{2p-2} - F_{2p-4}  &-1    & F_{2p-2}   &\cdots     &F_{2p-2}& 0&\cdots &0 \cr
\bar{e}_{\pi_{p}}  & F_{2p-6} - (s_1+2)F_{2p-4}  &s_2+2 & -F_{2p-4}  & \cdots & -F_{2p-4}   &1& \cdots &1\cr
\bar{e}_{\sigma_{(1,1)}} & -1                    & 0     & -2        & 0 &  0   & 
\cdots&\cdots& 0 \cr
\vdots             & \vdots                      & \vdots&0          & -2 &     \vdots &\vdots&  &\vdots  \cr
\bar{e}_{\sigma_{(1,s_1)}} & -1                  & 0     &           & 0 &    \ddots&0& 0&  \cr
\bar{e}_{\sigma_{(2,1)}} & 0                     & -1    &\vdots     & \vdots &    0 &-2& 0& 
\vdots \cr
\vdots             & \vdots                      & \vdots &          &  &    \vdots &\vdots&-2  &0  \cr
\bar{e}_{\sigma_{(2,s_2)}} & 0                   & -1     &0        & 0 &    0&0& 0& -2\cr
}
\end{equation}

\medskip
\noindent
This essentially completes {\sf Step 1}. However, to ease readability in $M$, we make some substitutions 
\begin{align*}
 d &=M_{1,1}= (s_1+2)F_{2p-2}-F_{2p-4} \\
 e &=M_{2,1}= F_{2p-6}-(s_1+2)F_{2p-4}\\
 f &=M_{2,2}= s_2+2\\
 g &= F_{2p-2}\\
 h &= -F_{2p-4}
\end{align*}
and also use some block matrix and vector abbreviations
\def\onesvector{\mathbf{1}}
\def\zerovector{\mathbf{0}}
\begin{align*}
    I_m&:=\text{ identity matrix of size }m \times m,\\
    J_{m,n}&:=\text{ all ones matrix of size }m \times n,\\
    \bigzero_{m,n}&:=\text{ all zero matrix of size }m \times n,\\
    \onesvector_m&:=\text{ all ones (column) vector of length }m,\\
    \zerovector_m&:=\text{ all zero (column) vector of length }m.
\end{align*}

Then the above matrix $M$ becomes
\begin{equation}
\label{M-abbreviated}
M=\begin{pmatrix}
d & \rvline & {\color{red}-1} & \rvline & g\cdot\onesvector^T_{s_1} & \rvline & \zerovector^T_{s_2}  \\
\hline
e & \rvline & f & \rvline & h\cdot\onesvector^T_{s_1} &\rvline & \onesvector^T_{s_2}  \\
\hline
-\onesvector_{s_1} & \rvline & \zerovector_{s_1} & \rvline & -2 I_{s_1}& \rvline & \bigzero_{s_1,s_2} \\
\hline
\zerovector_{s_2} & \rvline & -\onesvector_{s_2} & \rvline & \bigzero_{s_2,s_1}& \rvline & -2\cdot I_{s_2} 

\end{pmatrix}
\end{equation}
where we have emphasized in red the entry $M_{1,2}=-1$ that will be used in the first part of {\sf Step 2}.

\subsection{Step 2.}
 Since column vectors of $M$ give relations among
 the elements of \eqref{leaves-plus-path-endpoints} that are 
valid within $K=K(\Cone(T(p,s_1,s_2)))$,
by performing column operations, one can create more such valid relations.  We now use such column operations to try and create relations valid among only the elements of 
\eqref{leaf-images},
not involving the first two elements $
\bar{e}_{\pi_{p-1}},
\bar{e}_{\pi_p}$, as follows.

Use entry $M_{1,2}=-1$ as a pivot to clear row $1$,
eliminating entries $M_{1,1},$ and $M_{1,3},M_{1,4},\ldots,M_{1,s_1+2}$.  That is, add $d$ times column 2 to column 1, and add $g$ times column 2 to columns $3,4,\ldots, s_1+2$, giving this matrix $\hat M$:
 
\begin{equation}
\hat M=
\begin{pmatrix}
0 & \rvline & -1 & \rvline & \zerovector^T_{s_1} & \rvline & \zerovector^T_{s_2}  \\
\hline
df+e & \rvline & f & \rvline & (gf+h)\cdot\onesvector^T_{s_1} &\rvline & \onesvector^T_{s_2} \\
\hline
-\onesvector_{s_1} & \rvline & \zerovector_{s_1} & \rvline & -2 I_{s_1}& \rvline & \bigzero_{s_1,s_2} \\
\hline
-d \cdot \onesvector_{s_2} & \rvline & -\onesvector_{s_2} & \rvline & -g J_{s_2,s_1}& \rvline & -2\cdot I_{s_2} 
\end{pmatrix}.
\end{equation}

Prior to the next set of column operations, we rewrite $\hat{M}$ so as to isolate and emphasize in red its entry $\hat{M}_{2,s_1+3}=1$, which will be used next step to clear out almost all of row $2$:

\begin{equation*}    
\hat{M}=
\begin{pmatrix}
0 & \rvline & -1 & \rvline & \zerovector^T_{s_1} & \rvline &0 &\rvline&\zerovector^T_{s_2-1}  \\
\hline
df+e & \rvline & f & \rvline & (gf+h)\cdot\onesvector^T_{s_1} &\rvline & 
{\color{red}1}& \rvline &\onesvector^T_{s_2-1} \\
\hline
-\onesvector_{s_1} & \rvline & \zerovector_{s_1} & \rvline & -2 I_{s_1}& \rvline & \zerovector_{s_1}&\rvline&\bigzero_{s_1,s_2-1} \\
\hline
-d  & \rvline & -1 & \rvline & -g \cdot \onesvector^T_{s_1}& \rvline &-2&\rvline& \zerovector^T_{s_2-1} \\
\hline
-d\cdot \onesvector_{s_2-1} & \rvline & -\onesvector_{s_2-1} & \rvline & -g \cdot J_{s_2-1,s_1}& \rvline &\zerovector_{s_2-1} &\rvline& -2\cdot I_{s_2-1} 

\end{pmatrix}
\end{equation*}

\noindent
Now use that entry $\hat{M}_{2,s_1+3}=1$
to clear out the rest of row $2$ except for the entry $\hat{M}_{2,2}=f$, eliminating entries
$\hat{M}_{2,3},\hat{M}_{2,4},\ldots,\hat{M}_{2,s_1+2}$ and  $\hat{M}_{2,s_1+4},\hat{M}_{2,s_1+5},\ldots,\hat{M}_{2,s_1+s_2+2}$.
That is, subtract 
\begin{itemize}
\item $df+e$ times column $s_1+3$ from column $1$,
\item $gf+h$ times column $s_1+3$ from each of columns $3,4,\ldots,s_1+2$, and
\item $1$ times column $s_1+3$ from each of columns $s_1+4,s_1+5,\ldots,s+1+s_2+2$.
\end{itemize}

\noindent
The result is this matrix $\hat{\hat{M}}$:

\begin{equation*}    
\hat{\hat{M}}=
\begin{pmatrix}
0 & \rvline & -1 & \rvline & \zerovector^T_{s_1} & \rvline &0 &\rvline&\zerovector^T_{s_2-1}  \\
\hline

0 & \rvline & f & \rvline & \zerovector^T_{s_1} &\rvline & 
1& \rvline &\zerovector^T_{s_2-1} \\
\hline

-\onesvector_{s_1} & \rvline & \zerovector_{s_1} & \rvline & -2 I_{s_1}& \rvline & \zerovector_{s_1}&\rvline&\bigzero_{s_1,s_2-1} \\
\hline

2(df+e)-d  & \rvline & -1 & \rvline & (2(gf+h)-g) \cdot \onesvector^T_{s_1}& \rvline &-2&\rvline& 2\cdot \onesvector^T_{s_2-1} \\
\hline
-d\cdot \onesvector_{s_2-1} & \rvline & -\onesvector_{s_2-1} & \rvline & -g \cdot J_{s_2-1,s_1}& \rvline & \zerovector_{s_2-1} &\rvline& -2\cdot I_{s_2-1} 

\end{pmatrix}
\end{equation*}
Note that with the exception of columns $2$ and $s_1+3$ of $\hat{\hat{M}}$, its remaining columns all start with two zeroes, and therefore represent valid relations in $K$ among the elements of $\eqref{leaf-images}$.  Therefore deleting columns $2, s_1+3$ and rows $1,2$
from $\hat{\hat{M}}$ yields this 
 $(s_1 + s_2) \times (s_1 + s_2)$ matrix
$$
\begin{pmatrix*}
-\onesvector_{s_1} & \rvline & -2\cdot I_{s_1} & \rvline & \bigzero_{s_1,s_2-1} \\
\hline
2(df+e)-d  & \rvline & (2(gf+h)-g) \cdot \onesvector^T_{s_1} & \rvline & 2\cdot \onesvector^T_{s_2-1} \\
\hline
-d\cdot \onesvector_{s_2-1} & \rvline & -g\cdot J_{s_2-1,s_1} & \rvline & -2\cdot I_{s_2-1} 
\end{pmatrix*}
$$
whose columns represent relations in $K$ satisfied by
the elments of \eqref{leaf-images}.
Thus reordering its rows gives this matrix $M'$, whose columns represent relations in $K$ satisfied by a reordered list of
the elements of \eqref{leaf-images}:
\begin{equation}
\label{starting_matrix}
M'=
\begin{pmatrix*}
2(df+e)-d  & \rvline & (2(gf+h)-g) \cdot \onesvector^T_{s_1} & \rvline & 2\cdot \onesvector^T_{s_2-1} \\
\hline
-\onesvector_{s_1} & \rvline & -2\cdot I_{s_1} & \rvline & \bigzero_{s_1,s_2-1}  \\
\hline
-d\cdot \onesvector_{s_2-1} & \rvline & -g\cdot J_{s_2-1,s_1} & \rvline & -2\cdot I_{s_2-1}
\end{pmatrix*}
\end{equation}
This completes {\sf Step 2}.

\subsection{Step 3.}
We wish to show 
$\pm \det(M')=|K|=\tau(\Cone(T(p,s_1,s_2))=t(p,s_1,s_2)$,
starting with several row and column operations on $M'$ over $\mathbb{Z}$ that do not affect its determinant. In fact, these row and column operations will also be invertible over $\Z$ whenever $g$ is {\it even}, a fact which we will use later in {\bf Step 5}.

First, add $-\frac{g}{2}$ times each of columns $s_1+2,s_1+3,\ldots, s_1+s_2$ to each of columns $2,3,\ldots,s_1+1$, yielding:
    
\begin{equation*}\label{M1}
\begin{pmatrix*}

2(df+e)-d  & \rvline & (2(gf+h)-s_2g) \cdot \onesvector^T_{s_1} & \rvline & 2\cdot \onesvector^T_{s_2-1} \\
\hline
-\onesvector_{s_1} & \rvline & -2\cdot I_{s_1} & \rvline & \bigzero_{s_1,s_2-1}  \\
\hline
-d\cdot \onesvector_{s_2-1} & \rvline & \bigzero_{s_2-1,s_1} & \rvline & -2\cdot I_{s_2-1}

\end{pmatrix*}
\end{equation*}

Next, add each of rows $s_1+2,s_1+3,\ldots,s_1+s_2$ to the first row, yielding:
        
\begin{equation*}\label{M2}
\begin{pmatrix*}

2(df+e)-s_2d  & \rvline & (2(gf+h)-s_2g) \cdot \onesvector^T_{s_1} & \rvline &  \zerovector^T_{s_2-1} \\
\hline
-\onesvector_{s_1} & \rvline & -2\cdot I_{s_1} & \rvline & \bigzero_{s_1,s_2-1}  \\
\hline
-d\cdot \onesvector_{s_2-1} & \rvline & \bigzero_{s_2-1,s_1} & \rvline & -2\cdot I_{s_2-1}

\end{pmatrix*}
\end{equation*}

Then add $(gf+h-s_2\frac{g}{2})$ times each of rows $2,3,\ldots,s_1+1$ to row $1$.  The result is the following block lower-triangular matrix
        
\begin{equation}
\label{block-upper-triangular-matrix}
M''=\begin{pmatrix*}
x(p) & \rvline & \zerovector^T_{s_1} & \rvline &  \zerovector^T_{s_2-1} \\
\hline
-\onesvector_{s_1} & \rvline & -2\cdot I_{s_1} & \rvline & \bigzero_{s_1,s_2-1}  \\
\hline
-d\cdot \onesvector_{s_2-1} & \rvline & \bigzero_{s_2-1,s_1} & \rvline & -2\cdot I_{s_2-1} 
\end{pmatrix*}
\end{equation}
whose upper left entry $M''_{1,1}=x(p)$ has these formulas:
\begin{align} 
\label{formula-for-x}
 \notag   x(p)&:= 2(df+e) - s_2d-s_1(gf+h-\frac{1}{2}s_2g),\\
    &= (2s_1+2s_2+\displaystyle\frac{1}{2}s_1s_2+8)F_{2p-2}-(s_1+s_2+8)F_{2p-4}+2F_{2p-6}.
\end{align}
Note that this gives us an expression for $\det(M')$, as 
$$
\det(M')=\det(M'')=(-2)^{s_1+s_2-1} x(p).
$$
Recall that our goal in this {\sf Step 3} was to show that
$$
\det(M')=\pm t(p,s_1,s_2)=\pm 2^{s_1+s_2-2} a(p)
$$
where $a:=a(p)$ was defined was defined in \eqref{formula-for-a} as
\begin{equation}
\label{formula-for-a-repeated}
a:=a(p):=2^{1-s_1}(2\cdot b_{2p-3} +s_2b_{2p-4}).
\end{equation}
Hence one sees that {\sf Step 3} would be completed by proving the following lemma.

\begin{lemma}
\label{a-is-2x}
Let $p \geq 2.$ The two quantities $x(p)$ from \eqref{formula-for-x} and $a(p)$ from \eqref{formula-for-a-repeated} are related by
$x(p) = \displaystyle\frac{a(p)}{2}.$
\end{lemma} 

\begin{proof}
Moving forward we let 
\begin{align}
    A & = 2s_1+2s_2+\displaystyle\frac{1}{2}s_1s_2+8.\notag\\
    B &= -(s_1+s_2+8)\notag.\\
    C &= 2^{1-s_1}.\notag\\
D &= 2^{-s_1}s_2.\notag
\end{align}
Observe that if we let
\begin{align}
    x'(p)&= AF_{2p-4}+BF_{2p-6}+2F_{2p-8}\notag\\
    l(p)&=Cb_{2p-5}+Db_{2p-6}\notag\\
y'(p)&= AF_{2p-3}+BF_{2p-5}+2F_{2p-7}\notag\\
m(p)&=Cb_{2p-4}+Db_{2p-5}.\notag
\end{align}
The formulas \eqref{formula-for-x}, \eqref{formula-for-a-repeated} show that
the two sides of the lemma can both be rewritten as sums of two terms 
\begin{align}   
x(p)&=x'(p)+y'(p)\\
\frac{a(p)}{2}&=l(p)+m(p)
\end{align} 
hence, the lemma asserts the following:
\begin{equation}
   x'(p)+y'(p) =l(p)+m(p) \label{x'_y'} \hspace{1cm}
%\text{iff}\\  x(p) &= \displaystyle\frac{a(p)}{2}\notag.
\end{equation}
Thus it suffices to prove simultaneously by induction on $p$ that

\begin{align}    
x'(p) &= l(p) \label{x and l} \hspace{1cm} \text{and}\\
y'(p) &= m(p) \label{y and m}  
\end{align}
Base case: We check when $p = 2, p=3, p=4$ we see that 
\begin{align}  
x'(2) &= l(2) = s_1 +s_2 + 2.\notag\\
y'(2) &= m(2)   = s_1+ s_2+\displaystyle\frac{1}{2}s_1s_2+4\notag\\
x'(3) &= l(3) = 2s_1 +2s_2 +\displaystyle\frac{1}{2}s_1s_2+6\notag\\
y'(3) &= m(3)   = 3s_1+ 3s_2+s_1s_2+10\notag\\
x'(4) &= l(4) = 16+5s_1 +5s_2 + \displaystyle\frac{3}{2}s_1s_2.\notag\\
y'(4) &= m(4)   = 26+8s_1+8s_2+\displaystyle\frac{5}{2}s_1s_2\notag\\
\end{align}
Now assume that the statements  \eqref{x and l} and \eqref{y and m}
are true for $p = k, k-1, k-2.$ That is, 

\begin{align}
 AF_{2k-4}+BF_{2k-6}+2F_{2k-8}
   &=Cb_{2k-5}+Db_{2k-6} \label{ABCD_1}
   \end{align} and 
 \begin{align} 
    AF_{2k-3}+BF_{2k-5}+2F_{2k-7}&=Cb_{2k-4}+Db_{2k-5}.\label{ABCD_2}
\end{align} 

Then the sum of equations \eqref{ABCD_1}  and \eqref{ABCD_2}  is true and equal to  \eqref{x'_y'} evaluated at $k+1:$  
\begin{align}
 AF_{2k-2}+BF_{2k-4}+2F_{2k-6}
   &=Cb_{2k-3}+Db_{2k-4}.\notag
   \end{align}
Next our inductive step yields the following:

\begin{align}
 Cb_{2k-2}+Db_{2k-2}&=Cb_{2k-3}+Cb_{2k-4}+Db_{2k-4}+Db_{2k-5}\notag\\
 &=Cb_{2k-4}+Cb_{2k-5}+Cb_{2k-4}+Db_{2k-5}+Db_{2k-6}+Db_{2k-5}\notag\\
 &=AF_{2k-4}+BF_{2k-6}+2F_{2k-8}\notag\\
 & \qquad +AF_{2k-3}+BF_{2k-5}+2F_{2k-7}+AF_{2k-3}+BF_{2k-5}+2F_{2k-7}\notag\\
 &=A(F_{2k-4} + F_{2k-3}+F_{2k-3})\notag\\
 &\notag 
 \qquad +B(F_{2k-5}+ F_{2k-6} +F_{2k-5})+2(F_{2k-7}+F_{2k-8}+F_{2k-7})\\
 &=AF_{2k-1}+BF_{2k-3}+2F_{2k-5} \notag
   \end{align}
   which is equal to \eqref{ABCD_1} evaluated at $k+1$.
\end{proof}

This completes {\sf Step 3}.

\subsection{Step 4.}
This step was explained fully in the proof outline.

\subsection{Step 5.}

We now determine the structure of the sandpile group 
$$
K=K(\Cone(T(p,s_1,s_2)) \cong \Z^{s_1+s_2}/\im(M')
$$
where we recall that $M'$ was defined  as follows in
\eqref{starting_matrix}:
$$
M'=
\begin{pmatrix*}
2(df+e)-d  & \rvline & (2(gf+h)-s_2g) \cdot \onesvector^T_{s_1} & \rvline & 2\cdot \onesvector^T_{s_2-1} \\
\hline
-\onesvector_{s_1} & \rvline & -2\cdot I_{s_1} & \rvline & \bigzero_{s_1,s_2-1}  \\
\hline
-d\cdot \onesvector_{s_2-1} & \rvline & -g\cdot J_{s_2-1,s_1} & \rvline & -2\cdot I_{s_2-1}
\end{pmatrix*}.
$$
The analysis follows the cases in Theorem~\ref{sandpile-group-theorem},
based on $p \bmod{3}$ and the parity of $s_1, s_2$.

\vskip.1in
\noindent
{\sf Case 1:}  $p\equiv 1 \bmod{3}.$

Recall Fibonacci numbers $F_m$ are even when $m \equiv 0 \bmod{3}$, and are odd when $m \equiv 1,2 \bmod{3}$.  With this one can check that in Case 1 where $p \equiv 1 \bmod{3}$, one has
\begin{align*}
g&=F_{2p-2}\text{ is {\it even}, and }\\   
d&= (s_1+2)F_{2p-2}-F_{2p-4}\text{ is {\it odd}}.
\end{align*}
Therefore, as noted in {\sf Step 3}, since $g$ is even, the row and column performed there on $M'$ were also invertible over $\Z$ (they added a $\frac{g}{2}$ multiple of a row to another), resulting in this block triangular matrix 
$M''$:

$$
M''=\begin{pmatrix*}
x(p) & \rvline & \zerovector^T_{s_1} & \rvline &  \zerovector^T_{s_2-1}\\
\hline
-\onesvector_{s_1} & \rvline & -2\cdot I_{s_1} & \rvline & \bigzero_{s_1,s_2-1}  \\
\hline
-d\cdot \onesvector_{s_2-1} & \rvline & \bigzero_{s_2-1,s_1} & \rvline & -2\cdot I_{s_2-1} 
\end{pmatrix*}
$$

We next bring $M''$ to its Smith normal form via row and column operations invertible over $\Z$.  As $d$ is odd, one can subtract $\frac{d-1}{2}$ times each of columns $s_1+2,s_1+3,\ldots,s_1+s_2$ from column $1$ column, giving
\begin{align*}
\begin{pmatrix*}
x(p) & \rvline & \zerovector^T_{s_1} & \rvline &  \zerovector^T_{s_2-1}& \\
\hline
-\onesvector_{s_1} & \rvline & -2\cdot I_{s_1} & \rvline & \bigzero_{s_1,s_2-1} & \\
\hline
- \onesvector_{s_2-1} & \rvline & \bigzero_{s_2-1,s_1} & \rvline & -2\cdot I_{s_2-1}& 
\end{pmatrix*}
&=
\begin{pmatrix*}
x(p)  & \rvline &  \zerovector^T_{s_1+s_2-1}& \\
\hline
-\onesvector_{s_1 +s_2-1} &\rvline& -2\cdot I_ {s_1+s_2-1}& \\
\end{pmatrix*}\\
&=
\begin{pmatrix*}
x(p) & \rvline & 0 & \rvline &  \zerovector^T_{s_1+s_2-2}& \\
\hline
{\color{red}-1} & \rvline & -2 & \rvline & \zerovector^T_{s_1+s_2-2} & \\
\hline
-\onesvector_{s_1+s_2-2} & \rvline & \zerovector_{s_1+s_2-2} & \rvline & -2\cdot I_{s_1+s_2-2}& 
\end{pmatrix*}
\end{align*}
where the last rewriting emphasizes the red entry of $-1$ to be used in the next step.
Use this entry to clear out all other entries in the first column, that is, add $x(p)$
times row $2$ to row $1$, and subtract row $2$ from each of rows $3,4,\ldots,s_1+s_2$.  The result is this:
\begin{equation*}
\begin{pmatrix*}

0 & \rvline & -2x(p)  & \rvline &  \zerovector^T_{s_1+s_2-2}& \\
\hline
-1 & \rvline & -2& \rvline & \zerovector^T_{s_1+s_2-2} & \\
\hline
\zerovector_{s_1+s_2-2} & \rvline & 2\cdot \onesvector_{s_1+s_2-2} & \rvline & -2\cdot I_{s_1+s_2-2}& 
\end{pmatrix*}
\end{equation*}
Now subtract twice column $1$ from column $2$,
and add each of columns $3,4,\ldots,s_1+s_2$ to column $2$, giving:
$$
\begin{pmatrix*}
0 & \rvline & -2x(p)  & \rvline &  \zerovector^T_{s_1+s_2-2}& \\
\hline
-1 & \rvline & 0& \rvline & \zerovector^T_{s_1+s_2-2} & \\
\hline
\zerovector_{s_1+s_2-2} & \rvline &  \zerovector_{s_1+s_2-2} & \rvline & -2\cdot I_{s_1+s_2-2}& 
\end{pmatrix*}.
$$
After negating the whole matrix, and 
rearranging rows and columns, one reaches the Smith normal form:
$$
\begin{pmatrix*}
1 & \rvline&\zerovector^T_{s_1+s_2-2}  & \rvline& 0 \\ \hline
\zerovector_{s_1+s_2-2} & \rvline& 2\cdot I_{s_1+s_2-2} & \rvline& \zerovector_{s_1+s_2-2} \\ \hline
0 & \rvline& \zerovector^T_{s_1+s_2-2}& \rvline& 2x(p)
\end{pmatrix*}.
$$ 
This finally allows one to compute the structure of $K$:
\begin{align*}
K\cong \Z^{s_1+s_2}/\im(M') 
&\cong \Z^{s_1+s_2}/\im(M'') \\
&\cong \Z_2^{s_1+s_2-2} \oplus \Z_{2x(p)} \\
&= \Z_2^{s_1+s_2-2} \oplus \Z_{a},
\end{align*}
where the last equality used Lemma~\ref{a-is-2x}.
This agrees with Theorem~\ref{sandpile-group-theorem} in Case 1.

\vskip.1in
\noindent
{\sf Case 2:}   $p \not\equiv 1 \bmod{3},$ in particular $p \equiv 0 \bmod{3},$ and $s_1$ odd or $p \equiv 2 \bmod{3},$ and $s_1$ even.
\medskip

Again recall that Fibonacci numbers $F_m$ are even when $m \equiv 0 \bmod{3}$, and odd when $m \equiv 1,2 \bmod{3}$. Therefore in this case, when $p \not\equiv 1 \bmod 3$, one has that 
\begin{align*}
g&=F_{2p-2}\text{ is {\it odd}, however}\\
d&\text{ can be {\it even} or {\it odd.}}
\end{align*}

\noindent
We again begin with the matrix $M' \in \Z^{(s_1+s_2) \times (s_1+s_2)}$ from \eqref{starting_matrix} having $K \cong \Z^{s_1+s_2}/\im(M')$, copied below:
\begin{equation*}
M'=\begin{pmatrix}
2(df+e)-d & \rvline & (2(gf+h)-g)\onesvector^T_{s_1} & \rvline & 2\cdot\onesvector^T_{s_2-1} \\
\hline
-\onesvector_{s_1} & \rvline & -2I_{s_1} & \rvline & \bigzero_{s_1,s_2-1} \\
\hline
-d \cdot \onesvector_{s_2-1} & \rvline & -g J_{s_2-1,s_1} & \rvline & -2 I_{s_2-1} 
\end{pmatrix}
\end{equation*}
We will bring $M'$ to Smith normal form via
row and column operations invertible over $\Z$.
Note that we know
$$
\det(M')=\pm t(p,s_1,s_2)=\pm 2^{s_1+s_2-2} a(p),
$$
and what remains in order to prove Theorem~\ref{sandpile-group-theorem} in this case is to show that its Smith normal form is either
$$
\begin{cases}
  \diag(1,1,2^{s_1+s_2-3},2a(p)) & \text{ if }s_1\text{ or }s_2\text{ is odd},\\  
  \diag(1,1,2^{s_1+s_2-4},4,a(p)) & \text{ if }s_1,s_2\text{ are both even}.\
\end{cases}
$$
Here are the row and column operations.
\begin{enumerate}
\item First, add each of the last $(s_2-1)$ rows to the top row, yielding
$$
\begin{pmatrix}
2(df+e)-s_2d & \rvline & (2(gf+h)-s_2g )\onesvector^T_{s_1} & \rvline & \zerovector^T_{s_2-1} \\
\hline
-\onesvector_{s_1} & \rvline & -2I_{s_1} & \rvline & \bigzero_{s_1,s_2-1} \\
\hline
-d \cdot \onesvector_{s_2-1} & \rvline & -g J_{s_2-1,s_1} & \rvline & -2 I_{s_2-1} 
\end{pmatrix}
=
\begin{pmatrix}
z & \rvline & y \cdot \onesvector^T_{s_1} & \rvline & \zerovector^T_{s_2-1} \\
\hline
-\onesvector_{s_1} & \rvline & -2I_{s_1} & \rvline & \bigzero_{s_1,s_2-1} \\
\hline
-d \cdot \onesvector_{s_2-1} & \rvline & -g J_{s_2-1,s_1} & \rvline & -2 I_{s_2-1} 
\end{pmatrix},
$$
where we have abbreviated

\begin{align}
\label{definition-for-y}
y&:=2(gf+h)-s_2g
= (s_2 + 4)F_{2p-2} - 2F_{2p-4}\\
\label{definition-for-z}
z&:=2(df+e)-s_2d
=(4s_1 + 2s_2 + s_1s_2+ 8)F_{2p-2}+(-2s_1 - s_2 - 8)F_{2p-4}+ 2F_{2p-6}.
\end{align}

\item Using the fact that $g$ is odd, subtract $\frac{(g-1)}{2}$ times each of columns $s_1+2,s_1+3,\ldots,s_1+s_2$ from 
each of columns $2,3,\ldots, s_1+1$, yielding this matrix $\bar{M}$
\begin{equation}
\label{eq:promising-three-block-matrix}
\bar{M}:=\begin{pmatrix}
z & \rvline & y \cdot\onesvector^T_{s_1} & \rvline & \zerovector^T_{s_2-1} \\
\hline
-\onesvector_{s_1} & \rvline & -2I_{s_1} & \rvline & \bigzero_{s_1,s_2-1} \\
\hline
-d \cdot \onesvector_{s_2-1} & \rvline & - J_{s_2-1,s_1} & \rvline & -2 I_{s_2-1} 
\end{pmatrix}.
\end{equation}

\item One can rewrite $\bar{M}$ to emphasize its
$(2,1)$-entry of $-1$ in red that will be used as a pivot:
$$
\bar{M}:=\begin{pmatrix}
z & \rvline & y \cdot\onesvector^T_{s_1} & \rvline & \zerovector^T_{s_2-1} \\
\hline
-\onesvector_{s_1} & \rvline & -2I_{s_1} & \rvline & \bigzero_{s_1,s_2-1} \\
\hline
-d \cdot \onesvector_{s_2-1} & \rvline & - J_{s_2-1,s_1} & \rvline & -2 I_{s_2-1} 
\end{pmatrix}
=
\begin{pmatrix}
z & \rvline &y&\rvline& y \cdot\onesvector^T_{s_1-1} & \rvline & \zerovector^T_{s_2-1} \\
\hline
{\color{red}-1} & \rvline &-2 &\rvline & \zerovector_{s_1-1}^T& \rvline&\zerovector_{s_2-1}^T  \\\hline
-\onesvector_{s_1-1} & \rvline & \zerovector_{s_1-1}&\rvline& -2I_{s_1-1} & \rvline & \bigzero_{s_1-1,s_2-1} \\
\hline
-d \cdot \onesvector_{s_2-1} & \rvline&-\onesvector_{s_2-1} &\rvline & - J_{s_2-1,s_1-1} & \rvline & -2 I_{s_2-1} 
\end{pmatrix}
$$
\item Use the red $(2,1)$-entry of $-1$ as a pivot to clear out all other entries in the first column, giving this:
$$
\begin{pmatrix}
0 & \rvline &y-2z&\rvline& y \cdot\onesvector^T_{s_1-1} & \rvline & \zerovector^T_{s_2-1} \\
\hline
{\color{red}-1} & \rvline &-2 &\rvline & \zerovector_{s_1-1}^T& \rvline&\zerovector_{s_2-1}^T  \\\hline
\zerovector_{s_1-1} & \rvline & 2\cdot\onesvector_{s_1-1}&\rvline& -2I_{s_1-1} & \rvline & \bigzero_{s_1-1,s_2-1} \\
\hline
\zerovector_{s_2-1} & \rvline&(2d-1)\cdot \onesvector_{s_2-1} &\rvline & - J_{s_2-1,s_1-1} & \rvline & -2 I_{s_2-1} 
\end{pmatrix}
$$
\item Use the same red $(2,1)$-entry of $-1$ as a pivot to clear out the second row.  After that, one can delete the entire second row and first column, giving this smaller matrix with isomorphic cokernel, and with one fewer copy of $1$ removed from its Smith normal form:
$$
\begin{pmatrix}
y-2z&\rvline& y \cdot\onesvector^T_{s_1-1} & \rvline & \zerovector^T_{s_2-1} \\\hline
2\cdot\onesvector_{s_1-1}&\rvline& -2I_{s_1-1} & \rvline & \bigzero_{s_1-1,s_2-1} \\
\hline
(2d-1)\cdot \onesvector_{s_2-1} &\rvline & - J_{s_2-1,s_1-1} & \rvline & -2 I_{s_2-1} 
\end{pmatrix}
$$
\item Add each of the last $(s_2-1)$ columns to column $1$, yielding this matrix with {\it no occurrences of $d$'s}:
$$
\begin{pmatrix}
y-2z&\rvline& y \cdot\onesvector^T_{s_1-1} & \rvline & \zerovector^T_{s_2-1} \\\hline
2\cdot\onesvector_{s_1-1}&\rvline& -2I_{s_1-1} & \rvline & \bigzero_{s_1-1,s_2-1} \\
\hline
-\onesvector_{s_2-1} &\rvline & - J_{s_2-1,s_1-1} & \rvline & -2 I_{s_2-1} 
\end{pmatrix}
$$

\medskip

\item Subtract column $s_1$ from each of columns $1,2,\ldots,(s_1-1)$, giving this:\\
%\begin{equation}
$\begin{pmatrix}
-2z & \rvline & \zerovector^T_{s_1-2} & \rvline &  y &\rvline &\zerovector^T_{s_2-1}  &\\
\hline
2\cdot\onesvector_{s_1-2} & \rvline & -2 \cdot I_{s_1-2} & \rvline & \zerovector_{s_1-2}&\rvline & \zerovector_{s_1-2,s_2-1} &\\
\hline
4  & \rvline & 2 \cdot \onesvector^T_{s_1-2} & \rvline & -2& \rvline & \zerovector^T_{s_2-1}   &\\
\hline
 \zerovector_{s_2-1} & \rvline & \zerovector_{s_2-1,s_1-2}& \rvline &-\onesvector_{s_2-1} &\rvline & -2 \cdot I_{s_2-1}&\\

\end{pmatrix}=$
$\begin{pmatrix}
-2z & \rvline & \zerovector^T_{s_1-2}  & \rvline &  y&\rvline &\zerovector^T_{s_2-2}  &\rvline&0\\
\hline
2\cdot\onesvector_{s_1-2} & \rvline & -2 \cdot I_{s_1-2} & \rvline & \zerovector_{s_1-2}&\rvline & \zerovector_{s_1-2,s_2-2}&\rvline &\zerovector_{s_1-2}\\
\hline
4  & \rvline & 2 \cdot \onesvector^T_{s_1-2} & \rvline & -2& \rvline & \zerovector^T_{s_2-2}   & \rvline &0\\
\hline
 \zerovector_{s_2-2} & \rvline & \zerovector_{s_2-2,s_1-2}& \rvline &-\onesvector_{s_2-2} &\rvline & -2 \cdot I_{s_2-2}& \rvline &\zerovector_{s_2-2}\\
\hline
0 & \rvline & \zerovector^T_{s_1-2}  & \rvline &{\color{red}-1}&\rvline &  \zerovector^T_{s_2-2}&\rvline & -2 &\\
\end{pmatrix}.$
%\end{equation}
\item Use the red entry of $-1$ in the last row and column $s_1$  as a pivot to clear out all other entries in column $s_1$, giving this:

\begin{equation*} 
\begin{pmatrix}
-2z & \rvline & \zerovector^T_{s_1-2}  & \rvline &  0&\rvline &\zerovector^T_{s_2-2}  &\rvline&-2y\\
\hline
2\cdot\onesvector_{s_1-2} & \rvline & -2 \cdot I_{s_1-2} & \rvline & \zerovector_{s_1-2}&\rvline & \zerovector_{s_1-2,s_2-2}&\rvline &\zerovector_{s_1-2}\\
\hline
4  & \rvline & 2 \cdot \onesvector^T_{s_1-2} & \rvline & 0& \rvline & \zerovector^T_{s_2-2}   & \rvline &4\\
\hline
 \zerovector_{s_2-2} & \rvline & \zerovector_{s_2-2,s_1-2}& \rvline &\zerovector_{s_2-2} &\rvline & -2 \cdot I_{s_2-2}& \rvline &2\cdot\onesvector_{s_2-2}\\
\hline
0 & \rvline & \zerovector^T_{s_1-2}  & \rvline &{\color{red}-1}&\rvline &  \zerovector^T_{s_2-2}&\rvline & -2 &\\
\end{pmatrix}.
\end{equation*}

%%%%%%
\item Add columns $(s_1+1)$ through column $(s_1+s_2-2)$ to the last column obtaining:
\begin{equation*} 
\begin{pmatrix}
-2z & \rvline & \zerovector^T_{s_1-2}  & \rvline &  0&\rvline &\zerovector^T_{s_2-2}  &\rvline&-2y\\
\hline
2\cdot\onesvector_{s_1-2} & \rvline & -2 \cdot I_{s_1-2} & \rvline & \zerovector_{s_1-2}&\rvline & \zerovector_{s_1-2,s_2-2}&\rvline &\zerovector_{s_1-2}\\
\hline
4  & \rvline & 2 \cdot \onesvector^T_{s_1-2} & \rvline & 0& \rvline & \zerovector^T_{s_2-2}   & \rvline &4\\
\hline
 \zerovector_{s_2-2} & \rvline & \zerovector_{s_2-2,s_1-2}& \rvline &\zerovector_{s_2-2} &\rvline & -2 \cdot I_{s_2-2}& \rvline &\zerovector_{s_2-2}\\
\hline
0 & \rvline & \zerovector^T_{s_1-2}  & \rvline &{\color{red}-1}&\rvline &  \zerovector^T_{s_2-2}&\rvline & -2 &\\
\end{pmatrix}.
\end{equation*}
\item Deleting column $s_1$ and last row
gives an $(s_1+s_2-2) \times (s_1+s_2-2) $ matrix $\bar{\bar{M}}$ with isomorphic cokernel, and having two fewer copies of $1$ in its Smith normal form compared to $M'$:
\begin{equation*}
\bar{\bar{M}}:=\begin{pmatrix}
-2z & \rvline & \zerovector^T_{s_1-2} & \rvline &  \zerovector^T_{s_2-2} &\rvline &-2y  &\\
\hline
2\cdot\onesvector_{s_1-2} & \rvline & -2 \cdot I_{s_1-2} & \rvline &\zerovector_{s_1-2,s_2-2} &\rvline & \zerovector_{s_1-2} &\\
\hline
4  & \rvline & 2 \cdot \onesvector^T_{s_1-2} & \rvline & \zerovector^T_{s_2-2}& \rvline & 4   &\\
\hline
 \zerovector_{s_2-2} & \rvline & \zerovector_{s_2-2,s_1-2}& \rvline &-2 \cdot I_{s_2-2}&\rvline & \zerovector_{s_2-2} &\\

\end{pmatrix}
\end{equation*}

\noindent Note that all entries of $\bar{\bar{M}}$ are even, and our goal is to show its Smith normal form is either
$$
\begin{cases}
  \diag(2^{s_1+s_2-3},2a(p)) & \text{ if }s_1\text{ or }s_2\text{ is odd},\\  
  \diag(2^{s_1+s_2-4},4,a(p)) & \text{ if }s_1,s_2\text{ are both even}.\
\end{cases}
$$
Thus we can instead start with this matrix $\frac{1}{2} \bar{\bar{M}}$ 
\begin{equation*}
\frac{1}{2} \bar{\bar{M}} =\begin{pmatrix}
-z & \rvline & \zerovector^T_{s_1-2} & \rvline &  \zerovector^T_{s_2-2} &\rvline &-y  &\\
\hline
\onesvector_{s_1-2} & \rvline & -I_{s_1-2} & \rvline &\zerovector_{s_1-2,s_2-2} &\rvline & \zerovector_{s_1-2} &\\
\hline
2  & \rvline &  \onesvector^T_{s_1-2} & \rvline & \zerovector^T_{s_2-2}& \rvline & 2   &\\
\hline
 \zerovector_{s_2-2} & \rvline & \zerovector_{s_2-2,s_1-2}& \rvline &-  I_{s_2-2}&\rvline & \zerovector_{s_2-2} &
\end{pmatrix}
\end{equation*}
which will have 
$$
\det
\left(\frac{1}{2} \bar{\bar{M}}\right)
=\frac{1}{2^{s_1+s_2-2}}\det(\bar{\bar{M}})
=\frac{1}{2^{s_1+s_2-2}}\det(M')
=a(p).
$$
Our revised goal is to show that $\frac{1}{2} \bar{\bar{M}}$ has Smith form is either
$$
\begin{cases}
  \diag(1^{s_1+s_2-3},a(p)) & \text{ if }s_1\text{ or }s_2\text{ is odd},\\  
  \diag(1^{s_1+s_2-4},2,\frac{a(p)}{2}) & \text{ if }s_1,s_2\text{ are both even}.\
\end{cases}
$$
Note that inside $\frac{1}{2} \bar{\bar{M}}$, the copy of $-I_{s_2-2}$ contained in last $s_2-2$ rows  and the columns $s_1,s_1+1,\ldots,s_1+s_2-3$
has no other nonzero entries anywhere in the same rows or columns.  Consequently, one can delete those rows and columns to obtain the following matrix $\bar{\bar{\bar{M}}}$ with isomorphic cokernel, and $s_2-2$ fewer copies of one in its Smith normal form:
\begin{equation*}
\bar{\bar{\bar{M}}} =\begin{pmatrix}
-z & \rvline & \zerovector^T_{s_1-2} & \rvline &-y  &\\
\hline
\onesvector_{s_1-2} & \rvline & -I_{s_1-2} &\rvline & \zerovector_{s_1-2} &\\
\hline
2  & \rvline &  \onesvector^T_{s_1-2} & \rvline & 2   &
\end{pmatrix}
\end{equation*}
It remains to show that $\bar{\bar{\bar{M}}}$ has Smith normal form either 
$$
\begin{cases}
  \diag(1^{s_1-1},a(p)) & \text{ if }s_1\text{ or }s_2\text{ is odd},\\  
  \diag(1^{s_1-2},2,\frac{a(p)}{2}) & \text{ if }s_1,s_2\text{ are both even}.\
\end{cases}
$$
Starting with $\bar{\bar{\bar{M}}}$, adding each of rows $2,3,\ldots,s_1-1$ to the last row gives this matrix:
\begin{equation*}
\begin{pmatrix}
-z & \rvline & \zerovector^T_{s_1-2} & \rvline &-y  &\\
\hline
\onesvector_{s_1-2} & \rvline & -I_{s_1-2} &\rvline & \zerovector_{s_1-2} &\\
\hline
s_1  & \rvline &  \zerovector^T_{s_1-2} & \rvline & 2   &
\end{pmatrix}
\end{equation*}
Now adding each of columns $2,3,\ldots,s_1-1$ to the first column gives this matrix:
\begin{equation*}
\begin{pmatrix}
-z & \rvline & \zerovector^T_{s_1-2} & \rvline &-y  &\\
\hline
\zerovector_{s_1-2} & \rvline & -I_{s_1-2} &\rvline & \zerovector_{s_1-2} &\\
\hline
s_1  & \rvline &  \zerovector^T_{s_1-2} & \rvline & 2   &
\end{pmatrix}
\end{equation*}
Inside this matrix, the copy of 
$-I_{s_1-2}$ in the middle rows and columns has no other nonzero entries anywhere in the same rows or columns.  Deleting those rows and columns, one obtain this $2 \times 2$ matrix $N$ with isomorphic cokernel, and
$s_1-2$ fewer copies of one in its Smith normal form:
\begin{equation*}
N:=\begin{pmatrix}
-z & -y  \\
s_1  & 2   
\end{pmatrix}
\end{equation*}
It remains to show that $N$ has Smith normal form either 
$$
\begin{cases}
  \diag(1,a(p)) & \text{ if }s_1\text{ or }s_2\text{ is odd},\\  
  \diag(2,\frac{a(p)}{2}) & \text{ if }s_1,s_2\text{ are both even}.\
\end{cases}
$$

\vskip.1in
\noindent
{\sf Case 2a.} One of $s_1$ or $s_2$ is odd.\\

By the symmetry of $s_1,s_2$ in the definition of $T(p,s_1,s_2)$, we may assume without loss of generality that $s_1$ is odd. As $N$ has
both $2,s_1$ as entries, the GCD of its entries is $1$, so its Smith form is
$$
\diag(1,\det(N)) = \diag(1,a(p)).
$$

\vskip.1in
\noindent
{\sf Case 2b.} Both 
 $s_1, s_2$ are even.\\
 
Recall from \eqref{definition-for-y} and \eqref{definition-for-z} that
 \begin{align*}
z &= (4s_1 + 2s_2 + s_1s_2+ 8)F_{2p-2}+(-2s_1 - s_2 - 8)F_{2p-4}+ 2F_{2p-6} \\
% &=2(df+e)-s_2d\\
y  &= (s_2 + 4)F_{2p-2} - 2F_{2p-4}.\notag\\
\end{align*}
From this it follows that $y,z$ are both even.  Hence $N$ contains all even entries, and has $2$ as an entry, so its Smith normal form is 
$$
\diag\left(2,\frac{\det(N)}{2}\right) 
= \diag\left(2,\frac{a(p)}{2}\right).
$$
 as desired.

\end{enumerate}

\section{Proof of Theorem \ref{comb-tree-theorem}}

As mentioned in the introduction, we will consider the rooted plane binary tree $T_p$ with $n:=2p-1$ vertices that is a {\it left comb}, meaning that every right child is a leaf.  Thus $T_p$ has $p$ vertices $\pi_1,\ldots,\pi_p$ on the spine, and $p-1$ right-children $\ell_1,\ldots,\ell_{p-1}$ of $\pi_2,\ldots,\pi_p$, respectively.  It has $p$ leaves in total since $\pi_1$ is also a leaf.
See $T_6$ below.

%\begin{figure}[H]
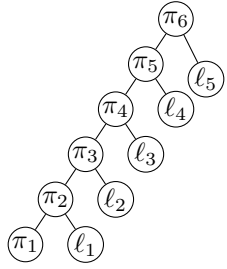
\begin{center}
\begin{tikzpicture}[scale=.4]

     \tikzstyle{every node}=[draw,circle,minimum size=1pt,inner sep=.5pt]
    
    %\node[color=red](v0) at (2,2){$v_0$};
    \node(v1) at (0,0){$\pi_1$};
    \node[](v2) at (1,1.5){$\pi_2$};
    \node[](v3) at (2,3){$\pi_3$};
    \node[](v4) at (3,4.5){$\pi_4$};
    \node[](v5) at (4,6){$\pi_5$};
    \node[](v6) at (5,7.5){$\pi_6$};
    \node[](v7) at (2,0){$\ell_1$};
    \node[](v8) at (3,1.5){$\ell_2$};
    \node[](v9) at (4,3){$\ell_3$};
    \node[](v10) at (5,4.5){$\ell_4$};
    \node[](v11) at(6,5.5){$\ell_5$};
    
    \draw(v1)--(v2)--(v3)--(v4);
    \draw (v4)--(v5)--(v6);
    %\draw(v0)--(v1);
    \draw(v2)--(v7);
    \draw(v3)--(v8);
    \draw(v4)--(v9);
    \draw(v5)--(v10);
    \draw(v6)--(v11);

  \tikzstyle{every node}=[]

  \end{tikzpicture}
\end{center}

\noindent
\textbf{Theorem} \ref{comb-tree-theorem}. 
    For any $p \geq 2$ one has $\mu(\Cone(T_p)) = 1$, but $\ell(T_p)=p$.
\begin{proof}
    Deleting the row and column of the cone vertex,  $\Cone(T_p)$ has this reduced Laplacian, $\overline{L}_{\Cone(T_p)}$:

\begin{equation}
\label{comb_laplacian}
\begin{blockarray}{}
\overline{L}_{\Cone(T_p)} = \bordermatrix{
~ & \bar{e}_{\pi_{1}} & \bar{e}_{\pi_{2}} & \bar{e}_{\pi_{3}} &  \cdots&\bar{e}_{\pi_{p}} & &\bar{e}_{\ell_{1}}& \bar{e}_{\ell_{2}} & \bar{e}_{\ell_{3}}&\cdots &\bar{e}_{\ell_{p-1}} \cr
~ & & &  &  &  & &  \cr
\bar{e}_{\pi_{1}}& 2   &  -1 &0     &\cdots& 0 &\rvline&0 & & \cdots &&0&\cr
\cline{8-12}
\bar{e}_{\pi_{2}}  & -1& 4 &-1  & \ddots   &\vdots & \rvline &&&\cr
\bar{e}_{\pi_{3}} &0 &   -1   & \ddots& \ddots &  0   &\rvline&&  & -I_{p-1}&\cr
\vdots & \vdots & \ddots&\ddots  & 4 & -1    &\rvline&  & \cr 
\bar{e}_{\pi_{p}} & 0 & \cdots   &0 & -1 & 3&\rvline& &\cr
\cline{2-12}
\bar{e}_{\ell_{1}} & &   &    &  &     &\rvline& & & \cr
\bar{e}_{\ell_{2}} &0  &   &    &  &     &\rvline& & \cr
\bar{e}_{\ell_{3}} & \vdots &    &-I_{p-1}     &  &     &\rvline& & &2I_{p-1}\cr
\vdots       & &  &      &  &     &\rvline&  &&  \cr
\bar{e}_{\ell_{p-1}} &0  &     &       &  &    &\rvline& &\cr
}
\end{blockarray}
\end{equation}

\noindent
Its rows and columns are indexed by the images of the standard basis vectors corresponding to $\pi_1,\ldots,\pi_p,\ell_1,\ldots,\ell_{p-1}$.
Recall that  
$
K(G):= \coker (\overline{L_G})\cong \mathbb{Z}^{n-1} / \im(\overline{L_G}) \cong \bigoplus_{i=1}^{n-1} \mathbb{Z} / \lambda_i\mathbb{Z}
$
where $\lambda_i$ are the diagonal entries in the Smith normal form of $\overline{L_G}$, so that $\lambda_i$ divides $\lambda_{i+1}$. Furthermore, if we denote by $\Delta_j$ the 
greatest common divisor of the determinants of all $(j \times j)-$minors of $L_G,$ then 
$
\Delta_j=\lambda_1 \lambda_2 \cdots 
\lambda_{j}.
$
%\frac{\Delta_{j}}{\Delta_{j-1}}.$  
Thus to show $K(\Cone(T_p))$ is cyclic, it suffices to exhibit two $(n-1) \times (n-1)$ minors whose gcd is 1,
since it implies $\lambda_1=\lambda_2=\cdots=\lambda_{n-1}=1$. This is provided by the following two claims. 
\vskip.1in
\noindent
{\sf Claim 1:} The minor of $\overline{L}_{\Cone(T_p)}$ obtained by deleting the first row and first column is odd.
\vskip.1in
\noindent
{\sf Claim 2:} The minor of $\overline{L}_{\Cone(T_p)}$ obtained by deleting the first row and last column is $\pm 2^{p-2}$. \\

\noindent {\sf Proof of Claim 1:} To show that the minor obtained by deleting the first row and first column of $\overline{L}_{\Cone(T_p)}$ is odd, we can reduce its entries modulo $2$, giving this result, which is easily seen to have determinant $1 \bmod{2}:$

\small
\begin{equation}
%\label{eq:promising-three-block-matrix}
\begin{pmatrix}
0&1 &0 &  \cdots &0& 0  &\rvline &\\
1&0 &1 & \cdots & 0& 0   &\rvline& \\
0 &1 &0 & & 0& 0  &\rvline& \\
\vdots & & &\ddots & & \vdots  &\rvline & I_{p-1} \\
0&0&0& \cdots &0& 1   &\rvline& \\
0&0&0 &\cdots &1 & 1    &\rvline& \\\hline
 & & &I_{p-1} & & &\rvline & \mathbf{0}_{p-1} 
\end{pmatrix}.
\end{equation}

\normalsize

%\vspace{1cm}
\noindent {\sf Proof of Claim 2;} Deleting the first row and last column of $\overline{L}_{\Cone(T_p)}$ gives this $(2p-2)\times (2p-2)$ matrix:

\begin{equation}\label{M'}
\begin{matrix}
 \bordermatrix{
~ & \bar{e}_{\pi_{1}} & \bar{e}_{\pi_{2}} & \bar{e}_{\pi_{3}} & \bar{e}_{\pi_{4}} &\cdots&\bar{e}_{\pi_{p}} &&\bar{e}_{\ell_{1}}& \bar{e}_{\ell_{2}} & \bar{e}_{\ell_{3}}&\cdots &\bar{e}_{\ell_{p-2}} \cr
~ & & &  &  &  & &  \cr
\bar{e}_{\pi_{2}}& -1   &  4 &-1 &0    &\cdots& 0 &\rvline\cr
\bar{e}_{\pi_{3}}  & 0& -1 &4  & \ddots   &\ddots& \vdots& \rvline &&&\cr
\bar{e}_{\pi_{4}} &0 &   \ddots   & \ddots& \ddots &  -1 & 0 &\rvline&&  & -I_{p-2}&\cr
\vdots & \vdots & \ddots&\ddots  & \ddots & 4   & -1&\rvline&  & \cr 
\bar{e}_{\pi_{p}} & \vdots & \cdots   &\cdots & 0 & -1& 3&\rvline&0 & \cdots &&\cdots&0\cr
\cline{2-13}
\bar{e}_{\ell_{1}} &\vdots  &   &    &  &     &&\rvline& & & \cr
\bar{e}_{\ell_{2}} &0  &   &    &  &     &&\rvline& & \cr
\bar{e}_{\ell_{3}} & \vdots &    &-I_{p-1}     &  &     &&\rvline& & &2I_{p-2}\cr
\vdots       &\vdots &  &      &  &     &&\rvline&  &&  \cr
\bar{e}_{\ell_{p-1}} &0  &     &       &  &    &&\rvline&0 & \cdots &&\cdots&0\cr
}
\end{matrix}
\end{equation}
\normalsize

\noindent
One can reorder the rows and columns of \eqref{M'} to obtain an upper triangular matrix, illustrated here for $T_6$: 

\small
\begin{equation*}
\bordermatrix{
~ & \bar{e}_{\pi_{1}} & \bar{e}_{\ell_{1}} & \bar{e}_{\pi_{2}} & \bar{e}_{\ell_{2}} & \bar{e}_{\pi_{3}} & \bar{e}_{\ell_{3}}& \bar{e}_{\pi_{4}} & \bar{e}_{\ell_{4}}& \bar{e}_{\pi_{5}} & \bar{e}_{\pi_{6}} \cr
\bar{e}_{\pi_{2}}&-1 &-1 &4 &0 &-1 &0 & & & &  \cr
\bar{e}_{\ell_{1}}& 0& 2& -1& 0& 0&0 & & & &  \cr
\bar{e}_{\pi_{3}} & & &-1 &-1 &4 &0 &-1 &0 & &  \cr
\bar{e}_{\ell_{2}}& & & 0& 2& -1& 0& 0&0 & &  \cr
\bar{e}_{\pi_{4}} & & &  & &-1 &-1 &4 &0 &-1 &0  \cr
\bar{e}_{\ell_{3}}& & & & & 0& 2& -1& 0& 0&0  \cr
\bar{e}_{\pi_{5}} & & & & &  & &-1 &-1 &4 &0  \cr
\bar{e}_{\ell_{4}}& & & & & & & 0& 2& -1& 0  \cr
\bar{e}_{\pi_{6}} & & & & & & &  & &-1 &3  \cr
\bar{e}_{\ell_{5}}& & & & & & & & & 0& -1&  \cr
}
\end{equation*}
\vspace{.3cm}
\normalsize
\noindent

Here is its general structure for $T_p$:
\vspace{.5cm}
\tiny
\begin{equation*}
\begin{blockarray}{}
 \bordermatrix{
~ & \bar{e}_{\pi_{1}} & \bar{e}_{\ell_{1}} & &\bar{e}_{\pi_{2}} & \bar{e}_{\ell_{2}}  & &\bar{e}_{\pi_{3}}&\bar{e}_{\ell_{3}}&&\bar{e}_{\pi_{4}}& \bar{e}_{\ell_{4}} & &\cdots&\cdots& &\cdots&\cdots& &\bar{e}_{\pi_{p-2}}&\bar{e}_{\ell_{p-2}}& &\bar{e}_{\pi_{p-1}}&\bar{e}_{\pi_{p}}\cr
~ & & &  &  &  & &\cr
\bar{e}_{\pi_{2}}& -1   &  -1 &\rvline&\phantom{-}4 &0 &\rvline&-1&0&\rvline&&&\rvline&&&\rvline&&&\rvline&&&\rvline&\cr
\bar{e}_{\ell_{1}}  & \phantom{-}0& \phantom{-}2 &\rvline&-1 & 0  &\rvline&\phantom{-}0 &0 &\rvline& &&\rvline&&&\rvline&&&\rvline&&&\rvline&\cr
\cline{2-24}
\bar{e}_{\pi_{3}} & &  &\rvline& -1&  -1 &\rvline&\phantom{-}4&\phantom{-}0&\rvline&-1&0  &\rvline&&&\rvline&&&\rvline&&&\rvline&\cr
\bar{e}_{\ell_{2}} &    &   &\rvline&\phantom{-}0 &\phantom{-}2 &\rvline&-1&\phantom{-}0 &\rvline&\phantom{-}0&0&\rvline&&&\rvline&&&\rvline&&&\rvline&\cr
\cline{2-24}
\bar{e}_{\pi_{4}} & &  &\rvline&&&\rvline&-1 & -1  &\rvline&\phantom{-}4&\phantom{-}0&\rvline&-1& \phantom{-}0&\rvline&&&\rvline&&&\rvline& \cr
\bar{e}_{\ell_{3}} &   &   &\rvline&&&\rvline&\phantom{-}0 & \phantom{-}2&\rvline&-1&\phantom{-}0 &\rvline&\phantom{-}0&\phantom{-}0&\rvline&&&\rvline&&&\rvline&\cr
\cline{2-24}
\bar{e}_{\pi_{5}} &  &  &\rvline& &   &\rvline& & &\rvline&   \ddots &\ddots&\rvline&\ddots & \ddots&\rvline&\ddots&\ddots&\rvline&&&\rvline&\cr
\bar{e}_{\ell_{4}} &  &   &\rvline&&  &\rvline&&&\rvline& \ddots& \ddots&\rvline&\ddots&\ddots&\rvline&\ddots&\ddots&\rvline&&&\rvline&\cr
\cline{2-24}
\vdots       & &      &\rvline& & &\rvline&  &&\rvline&&  &\rvline&\ddots&\ddots&\rvline&\ddots&\ddots&\rvline&\ddots&\ddots&\rvline&\cr
\vdots &  &     &\rvline&     &  &\rvline&&  &\rvline&&&\rvline&\ddots&\ddots&\rvline&\ddots&\ddots&\rvline&\ddots&\ddots&\rvline&\cr
\cline{2-24}
\vdots      & &      &\rvline& & &\rvline&  &&\rvline&&  &\rvline&&&\rvline&-1&-1&\rvline&\phantom{-}4&0&\rvline&-1&0\cr
\vdots &  &     &\rvline&     &  &\rvline&&  &\rvline&&&\rvline&&&\rvline&\phantom{-}0&\phantom{-}2&\rvline&-1&0&\rvline&\phantom{-}0&0\cr
\cline{2-24}
\bar{e}_{\pi_{p-1}}      & &      &\rvline& & &\rvline&  &&\rvline&&  &\rvline&&&\rvline&&&\rvline&-1&-1&\rvline&\phantom{-}4&-1\cr
\bar{e}_{\ell_{p-2}} &  &     &\rvline&     &  &\rvline&&  &\rvline&&&\rvline&&&\rvline&&&\rvline&\phantom{-}0&\phantom{-}2&\rvline&-1&\phantom{-}0\cr
\cline{2-24}
\bar{e}_{\pi_{p}}       & &      &\rvline& & &\rvline&  &&\rvline&&  &\rvline&&&\rvline&&&\rvline&&&\rvline&-1&\phantom{-}3\cr
\bar{e}_{\ell_{p-1}} &  &     &\rvline&     &  &\rvline&&  &\rvline&&&\rvline&&&\rvline&&&\rvline&&&\rvline&\phantom{-}0&-1&\cr
}
\end{blockarray}
\end{equation*}

\normalsize
Note that its determinant, which is the product of the diagonal entries, is $\pm 2^{p-2}$, proving Claim 2.
\end{proof}
 \newpage
\bibliographystyle{alpha}
\bibliography{ref}
    
\end{document}

%% file: drawing1.tex
    \begin{tikzpicture}

    \tikzstyle{every node}=[draw,circle,fill=white,minimum size=4pt,inner sep=2pt]
    \node[](v1) at (0,0){};
    \node[](v2) at (1,0){};
    \node[](v3) at (2,0){};
    \node[](v4) at (3,0){};
    \node[](v5) at (4,0){};
    \node[](v6) at (5,.75){};
    \node[](v7) at (5,-.75){};
    \node[](v8) at (5.2,-.5){};
    \node[](v9) at (5.2,.5){};
    
    \draw(v1)--(v2)--(v3)--(v4);
    
    \draw (v4)--(v5)--(v6);
    \draw(v9)--(v5)--(v7);
    \draw(v5)--(v8);
  
     \def\C{(5.75,1)} 
     \def\D{(5.75,-1)} 
     \def\E{(0,-.75)} 
     \def\F{(4,-.75)}

  \tikzstyle{every node}=[]
    
  \coordinate (A) at (6,0);
   \coordinate (B) at (2,-1);
  
 \draw[] (A) circle (0);

  \end{tikzpicture}

%% file: drawing2.tex
%    \begin{subfigure}
         % \resizebox{2in}{!}{
    \begin{tikzpicture}
    %\node at (-.9,.1){$\vdots$};
    %\node at (5.2,.1){$\vdots$};
    %\node[]at (2.52,0){$\cdots$};
    \tikzstyle{every node}=[draw,circle,fill=white,minimum size=4pt,inner sep=2pt]
    \node[](v1) at (0,0){};
    \node[](v2) at (1,0){};
    \node[](v3) at (2,0){};
    \node[](v4) at (3,0){};
    \node[](v5) at (4,0){};
    \node[](v6) at (5,.75){};
    \node[](v7) at (5,-.75){};
    \node[](v8) at (5.2,-.5){};
    \node[](v9) at (5.2,.5){};
    %\node[](v10) at (-1.2,.5){};
     \node[](v11) at (-1.0,.75){};
     \node[](v12) at (-.8,0){};
     \node[](v13) at (-1.0,.-.75){};
    
    \draw(v1)--(v2)--(v3)--(v4);
    
    \draw (v4)--(v5)--(v6);
    \draw(v9)--(v5)--(v7);
    \draw(v5)--(v8);
   % \draw(v10)--(v1);
    \draw(v11)--(v1);
    \draw(v12)--(v1);
    \draw(v13)--(v1);
     \def\C{(5.75,1)} 
     \def\D{(5.75,-1)} 
     \def\E{(0,-.75)} 
     \def\F{(4,-.75)} 
%  \draw[decorate,decoration={brace,mirror}] \D -- \C;
%\draw[decorate,decoration={brace,mirror}] \E -- \F;

  \tikzstyle{every node}=[]
    
  \coordinate (A) at (6,0);
   \coordinate (B) at (2,-1);
  
 \draw[] (A) circle (0);
 %\node[xshift=.65cm] at (A) {$s$ vertices};
 %\node[yshift=-.25cm] at (B) {$p$ vertices};

  \end{tikzpicture}
    % }

%\end{subfigure}
  

%% file: drawing3.tex
% \begin{subfigure}

\begin{tikzpicture}
  [scale=.4,auto=left,every node/.style={circle,fill=black!20}]

  %%%%%%
    \node (s7) at (-13,4) {$\sigma_{(1,2)}$};
   \node (s6) at (-10,4) {$\sigma_{(1,1)}$};
  \node (s8) at (-14,0) {$\sigma_{(1,3)}$};
  \node (s9) at (-13,-4) {$\sigma_{(1,4)}$};
  \node (s10) at (-10,-4) {$\sigma_{(1,5)}$};
 \node (p1) at (-10,0) {$\pi_1$};
  \node (p2) at (-7,0) {$\pi_2$};
  \node (p3) at (-4,0) {$\pi_3$};
  \node (p4) at (-1,0) {$\pi_4$};
  \node (p5) at (2,0) {$\pi_5$};
  \node (p6) at (5,0) {$\pi_6$};
  \node (s1) at (5,4) {$\sigma_{(2,1)}$};
  \node (s2) at (8,2.5) {$\sigma_{(2,2)}$};
  %\node (s3) at (9,0) {$\sigma_{2,3}$}; 
  \node (s4) at (8,-2.5) {$\sigma_{(2,3)}$}; 
  \node (s5) at (5,-4) {$\sigma_{(2,4)}$};
 
  \foreach \from/\to in {s6/p1,s7/p1,s8/p1,s9/p1,s10/p1,p6/p5,p5/p4,p4/p3,p3/p2,p2/p1,p6/s1,p6/s2,p6/s4,p6/s5}
    \draw (\from) -- (\to);
 
\end{tikzpicture}

%     \end{subfigure}
  